\documentclass[11pt,regno]{amsart}
\usepackage{euscript,graphicx,epstopdf,amscd,amsgen,amsfonts,amssymb,latexsym,
amsmath,amsthm,graphicx,mathrsfs,times,color,overpic}
\newtheorem{theorem}{Theorem}[section]
\newtheorem{lemma}[theorem]{Lemma}
\newtheorem{proposition}[theorem]{Proposition}

\newtheorem{corollary}[theorem]{Corollary}

\newtheorem{mtheo}{Theorem}

\theoremstyle{definition}

\newtheorem{remark}[theorem]{Remark}
\newtheorem*{conjecture}{Conjecture}
\numberwithin{equation}{section}

\newcommand{\eqdef}{\stackrel{\scriptscriptstyle\rm def}{=}}

\DeclareMathOperator{\diam}{diam}

\DeclareMathOperator{\dist}{dist}
\DeclareMathOperator{\interior}{int}

\DeclareMathOperator{\llangle}{\langle\hspace{-0.2cm}\langle}
\DeclareMathOperator{\rrangle}{\rangle\hspace{-0.2cm}\rangle}

\def\s{{\rm s}}
\def\u{{\rm u}}
\def\cs{{\rm cs}}
\def\cu{{\rm cu}}

\def\sW{\mathscr{W}}

\def\bN{\mathbb{N}}
\def\bZ{\mathbb{Z}}
\def\bR{\mathbb{R}}

\def\cF{\EuScript{F}}
\def\cC{\EuScript{C}}
\def\cK{\EuScript{K}}

\def\cR{\EuScript{R}}
\def\eS{\mathscr{S}}

\def\cT{\mathscr{T}}

\def\cH{\EuScript{H}}

\def\cM{\EuScript{M}}
\def\cE{\mathcal{E}}
\def\cL{\EuScript{L}}

\DeclareMathSymbol{\varnothing}{\mathord}{AMSb}{"3F}
\renewcommand{\emptyset}{\varnothing}

\begin{document}

\title[Non-hyperbolic behavior of geodesic flows]{Non-hyperbolic behavior of geodesic flows\\ of rank 1 surfaces}
\author[K.~Gelfert]{Katrin Gelfert}
\address{Instituto de Matem\'atica Universidade Federal do Rio de Janeiro, Av. Athos da Silveira Ramos 149, Cidade Universit\'aria - Ilha do Fund\~ao, Rio de Janeiro 21945-909,  Brazil}\email{gelfert@im.ufrj.br}

\begin{abstract}
	We prove that for the geodesic flow of a rank 1 Riemannian surface which is expansive but not Anosov the Hausdorff dimension of the set of vectors with only zero Lyapunov exponents  is large. 
\end{abstract}

\begin{thanks}
{KG  has been supported by CNPq (Brazil). She is very grateful for the comments by the referee.}
\end{thanks}

\maketitle
\tableofcontents
\section{Introduction}

Let $M$ be a simply connected compact Riemannian surface of nonpositive sectional curvature and negative Euler characteristic. We consider the geodesic flow $G=(g^t)_{t\in\bR}$ on the unit tangent bundle $T^1M$. In this paper we complete our analysis in~\cite{BurGel:14} providing information of the Hausdorff dimension of the following type of level sets: given $\alpha\ge0$ let
\begin{multline}\label{neeww}
\cL^+(\alpha) \eqdef \\\Big\{v\in T^1M\colon 
 	v\text{ is Lyapunov forward regular and has Lyapunov exponents }0,\pm\alpha
\Big\}
\end{multline}
(further equivalent characterizations will be given in Section~\ref{s:levset}). We also consider the subset $\cL^-(\alpha)$ where we assume that $v\in\cL^-(\alpha)$ is  Lyapunov \emph{backward} regular and has Lyapunov exponents $0,\pm\alpha$.

\begin{mtheo} \label{jen}
	Let $M$ be a connected compact Riemannian surface of nonpositive sectional curvature and negative Euler characteristic. Suppose that the geodesic flow on the unit tangent bundle is expansive but not Anosov.
     Let $\alpha_1$ be the positive Lyapunov exponent of the Liouville measure.
     
      Then for all $\alpha\in[0,\alpha_1]$ the level set $\cL^-(\alpha)\cap\cL^+(\alpha)$ is nonempty. More precisely, there is a dense set $D\subset T^1M$ such that for every $v\in D$ and for every $\varepsilon>0$ there exists a set $L^+\subset \cL^+(\alpha)\cap \sW^\u_\varepsilon(v)$ which satisfies
    \[
    \dim_{\rm H}(L^+) =1,
    \]
where $\sW^\u_\varepsilon(v)$ denote the set of points $w$ in the local unstable manifold of $v$ satisfying $\rho(w,v)\le\varepsilon$ (here $\rho$ denotes the distance function on $T^1M$ induced by the Riemannian metric). The analogous statement holds for the local stable manifold of $v$ and the set $\cL^-(\alpha)$.
\end{mtheo}

The multifractal analysis of various types of level sets, including Lyapunov exponents, in the case of flows started with Pesin and Sadovskaya~\cite{PesSad:01} in the case of a conformal flow on a uniformly hyperbolic
 set. Our results are essentially based on the thermodynamic formalism for equilibrium measures of H\"older continuous potentials for flows (see, in particular, Bowen and Ruelle~\cite{BowRue:75}). Similarly to the work in~\cite{PesSad:01} and~\cite{BowRue:75} we are going to replace the flow by an associated suspension flow over a subshift of finite type. Therefore we use the general approach by Bowen and Walters in~\cite{BowWal:72}, that allows a symbolic description for, in fact, \emph{any} fixed point-free flow, to model the non-hyperbolic geodesic flow on the unit tangent bundle. Using Bowen's method~\cite{Bow:73} we will derive additional regularity properties of the suspended flow on certain uniformly hyperbolic subsets.

The particular emphasis of this paper is on the exponent $\alpha=0$.
The hypothesis of Theorem~\ref{jen} for $\alpha\in(0,\alpha_1]$ is shown in~\cite{BurGel:14} in the general case dropping the assumption that the flow is expansive.
To facilitate our approach, we assume expansivity.
Work by Paternain~\cite{Pat:93} and Ghys~\cite{Ghy:84} implies that all expansive geodesic flows on a compact Riemannian surface $M$ are topologically equivalent, and in particular are topologically conjugate to an Anosov flow of a Riemannian metric with constant negative curvature (compare~\cite[Theorem 1.3]{LopRosRug:07}). However, this conjugacy is certainly not sufficient to derive any result about Hausdorff dimension (which is an invariant under bi-Lipschitz conjugacy only).

As our focus lies on vectors with Lyapunov exponents all being zero, let us briefly mention some further properties of geodesic curves that lack hyperbolic behavior. 
First recall that a vector $v\in T^1M$ has \emph{rank $1$} if there is no perpendicular parallel Jacobi field along the geodesic $\gamma_v$ tangent to $v$ (we provide more details in Section~\ref{ss:jf}). In the case of a surface (which we consider here), a vector has rank $1$ if and only if it is tangent to some geodesic that eventually passes through some point on the manifold where the Gaussian curvature is negative. Note that the set $\cR\subset  T^1M$  of all vectors tangent to rank $1$ geodesics is invariant under the geodesic flow and open and dense in $T^1M$. The complement $\cH\eqdef T^1M\setminus \cR$ is invariant and closed and nonempty since we assume that $G|_{T^1M}$ is not Anosov. 
As such, any recurrent geodesic that is tangent to a rank $1$ vector shows some  sensitive dependence on initial conditions similar to geodesics in negatively curved manifolds, while a geodesic tangent to a vector of higher rank in the complement $\cH=T^1M\setminus \cR$  lacks hyperbolic behavior and behaves like a geodesic in a flat space.   
Notice that the fact that there exist parallel perpendicular Jacobi fields along a geodesic tangent to $v\in  \cH$ results in the fact that the largest Lyapunov exponent, and hence all Lyapunov exponents, at $v$ are zero and thus that $\cH\subset\cL(0)$. 

Topological and measure theoretic properties of the sets $\cR$ and $\cH$ are the subject of numerous investigations. In particular, it is conjectured,  that $\cR$ has full Liouville measure and, hence, the geodesic flow is ergodic. Although this is true for all known examples, at the present state of the art it remains unproved in the general case. 
Using compactness and invariance of the set $\cH$ and the variational principle for the entropy, it follows immediately from the Ruelle inequality that the topological entropy satisfies $h({G}|_\cH)=0$. In fact, by~\cite[Theorem 1.3]{BurGel:14} we have the stronger result that for every $\alpha\in[0,\alpha_1]$ we have%
\footnote{Notice that the sets $\cL(\alpha)$ are in general non-compact and accordingly, we have to use the general concept of topological entropy introduced by Bowen (see for example~\cite{Pes:97}).}
\[
	h(G|_{\cL(\alpha)})=\alpha.
\]

We are aware of the fact that our hypothesis are quite specific. 
On the one hand we restrict our considerations to a surface $M$ since at the present state of the art dimension analysis is essentially restricted to conformal flows and hence our approach does not permit formulas for the Hausdorff dimension of level sets (even in the Anosov case) for higher-dimensional manifolds.
On the other hand, even though the symbolic coding of the flow in~\cite{BowWal:72} applies to any fixed point-free flow, we require expansivity (see Section~\ref{sec:exp}). We do so mainly in order to have a tempered distortion property (see~\eqref{distorrrrtion}) that is an essential ingredient in our approach. 
Note that the geodesic flow fails to be expansive if there exist bi-asymptotic geodesics in its Riemannian covering. 
Certainly, if there exist two (distinct) geodesics that are edges of a flat strip (an isometrically and totally embedded copy of $[0,r]\times\bR$, $r>0$) $\cH$ and thus $\cL(0)$ contain a set of vectors that has dimension at least two. But not much else is known about the topological and dynamical properties of these sets.  

Finally, not much else is known about fractal properties of the level set
\[
	\cL(0)
	\eqdef\cL^-(0)\cap\cL^+(0).
\]
By the methods presented in this paper, one can show that it locally contains a subset which is the (continuous image of a direct) product of the two sets $L^+$ and $L^-$ claimed in Theorem~\ref{jen} defined in terms of the local product structure (see Corollary~\ref{cor:product}). To conclude about the Hausdorff dimension of such a product, the only obstacle so far is the regularity of this structure which, in general, can be only H\"older continuous with some H\"older exponent away from one (see, for example~\cite{GerWil:99} and references therein). We state the following conjecture. 
\begin{conjecture}
	Under the hypotheses of Theorem~\ref{jen} we have
	 $\displaystyle \dim_{\rm H}(\cL(0))=3$.
\end{conjecture}

The paper is organized as follows. In Section~\ref{s:levset} we recall some preliminaries on geometry and Lyapunov exponents and exploit the conformal structure of the geodesic flow in order to derive equivalent characterizations of the level set~\eqref{neeww}. Section~\ref{sec:prelim} recalls some thermodynamic preliminaries.
In Section~\ref{s:sus} we consider a symbolic description of the flow $G|_{T^1M}$ by means of a suspension flow over a finite family of local cross sections. This family will be used as a reference for all further steps. Further, we take a family of basic sets that fill the non-hyperbolic set $T^1M$ and symbolically model each of them over the fixed family of cross sections. 
In order to analyze the exponent at the spectrum boundary $\alpha=0$, in Section~\ref{sec:bridge} we introduce the concept of a bridging measure on an abstract shift space. For such a measure typical points have prescribed Birkhoff averages (and in particular can characterize orbits with Lyapunov exponent $\alpha=0$) and prescribed limits of finite-time entropies that together allow for an estimate of its local Hausdorff dimension. Note that such a bridging measure is in general not invariant. Notice that any invariant measure of non-zero entropy and hence with non-zero exponents would not be able to capture properties of $\cL(0)$.
Finally, we conclude the proof of Theorem~\ref{jen} in Section~\ref{sec:locate}.

\section{Characterization of level sets}\label{s:levset}

In this section we provide a number of preliminary properties of the geodesic flow. Our main aim is to derive the following equivalent characterizations of any level set~\eqref{neeww}. Here the subspace $F^\u_v\subset T_vT^1M$ will be defined in Section~\ref{sec:invlinfie}, the potential $\varphi^{(\u)}$ will be defined in~\eqref{phidef}, and $\lVert L\rVert$ and $[L]=\lVert L^{-1}\rVert^{-1}$ denote the norm and the inverse norm of a linear operator $L$, respectively.

\begin{proposition}\label{pro:whams}
	For any $\alpha\ge0$ we have 
\[\begin{split}
		\cL^\pm(\alpha) 
		&=\Big\{v\in T^1M\colon
		\lim_{t\to\pm\infty} -\frac{1}{t}\int_0^t\varphi^{(\u)}(g^s(v))\,ds=\alpha\Big\}\\
		&= \Big\{v\in T^1M\colon 
 		\lim_{t\to\pm\infty}\frac{1}{t}\log\,\phi(t,v) =\alpha \Big\},
\end{split}\]
where $\phi(t,v)$ is either $\lVert d{g}^t_v\rVert, \lVert dg^t|_{F_v}\rVert, \lVert dg^t|_{F^\u_v}\rVert, [ dg^t|_{F_v}]^{-1},$ or $[ d{g}^t_v]^{-1}$. 	
\end{proposition}

To prove the above, we will essentially exploit the conformal structure of the flow on a three-dimensional manifold (see also~\cite{PesSad:01} for a discussion of conformal flows) and derive an almost multiplicativity property for the dynamical cocycle of the flow.
The proof of Proposition \ref{pro:whams} will be completed at the end of Section~\ref{ss:almost}.

\subsection{Geometry}\label{ss:jf}

We refer particularly to \cite[Chapter IV]{Bal:95} or to~\cite{Ebe:73b,Ebe:01} for this subsection. 
The geodesic flow $G=(g^t)_{t\in\bR}$ acts on the tangent bundle $TM$ by $g^t(v)=\dot\gamma_v(t)$, where $\gamma_v$ denotes the geodesic curve determined by $\dot\gamma_v(0)=v$. Note that 
\begin{equation}\label{eq:forbac}
	g^t(v)
	= -g^{-t}(-v).
\end{equation}

Given a vector $v\in T_pM$, we identify $T_vTM$ with  $T_pM\oplus T_pM$ via the isomorphism
\[
	\Psi\colon
	\xi\mapsto(d\pi(\xi),C(\xi)),
\]
where $\pi\colon TM\to M$ denotes the canonical projection and $C\colon TTM\to TM$ denotes the connection map defined by the Levi Civita connection. The components of the direct sum $TM\oplus TM$ are also referred to as \emph{horizontal} and \emph{vertical subspaces} (see for example~\cite{Bal:95,Ebe:73b}).

A \emph{Jacobi field} $J$ along a geodesic $\gamma\colon t\mapsto \gamma(t)\in TM$ is a vector field $J\colon t\mapsto J(t)\in T_{\gamma(t)}TM$ that satisfies the Jacobi equation
\[
 	J''(t) +R(J(t),\dot\gamma(t))\dot\gamma(t) = 0,
\]
where $R$ denotes the Riemannian curvature tensor of $M$ and $'$ denotes covariant differentiation along $\gamma$.  A Jacobi field $J$ along a geodesic $\gamma$ with $\dot\gamma(0)=v$  is uniquely determined by its initial conditions $(J(0), J'(0))\in T_{\pi v}M\oplus T_{\pi v}M$. Indeed, given $\xi\in T_vTM$, denote by $J_\xi$ be the unique Jacobi field along $\gamma_v$ such that $J_\xi(0)=d\pi_v(\xi)$ and $J_\xi'(0)=C_v(\xi)$.
Jacobi fields give a geometric description of the derivative of the geodesic flow. Given $\xi\in T_vTM$, we have $\Psi(dg^t_v(\xi))=(J_\xi(t),J_\xi'(t))$.

The Riemannian metric on $M$ lifts to the {\em Sasaki metric} on $TM$ induced by the scalar product structure, for every $\xi,\eta\in T_vTM$ defined by
\[
 \llangle \xi,\eta\rrangle
 = \langle d\pi_v (\xi),d\pi_v (\eta)\rangle_{\pi v}
 	+ \langle C_v(\xi),C_v(\eta)\rangle_{\pi v}.
 \]
This metric induces a distance function on $T^1M$ which we denote by $\rho$.

Recall that the geodesic flow leaves invariant the unit tangent bundle $T^1M$. Under the above defined isomorphism we have $\Psi(T_vT^1M)= T_pM\oplus v^\perp$, where $v^\perp$ is the subspace of $T_pM$ orthogonal to $v\in T^1M$. 

Denote by $V\subset TTM$ the vector field of the geodesic flow.
Note that for $v\in T^1M$ and $\xi\in T_vT^1M$ we have $\llangle\xi,V(v)\rrangle=0$ if and only if $\langle J_\xi(0),v\rangle_{\pi v}=0=\langle J_\xi'(0),v\rangle_{\pi v}$ if and only if $\langle J_\xi(t),\dot\gamma_v(t)\rangle_{\pi g^t(v)}=0$ for all $t\in \bR$ and hence 
\[
	\llangle dg^t_v(\xi),V(g^t(v))\rrangle=0
	\quad\text{ for all }\quad t\in\bR.
\]	 
Any such Jacobi field is called  an {\em orthogonal Jacobi field}. In words,  the set of Jacobi fields $J$ along a geodesic $\gamma$ such that $J(0)$ and $J'(0)$ are orthogonal to $\dot\gamma(0)$ is exactly the set of Jacobi fields $J$ such that $J(t)$ is normal to $\dot\gamma(t)$  for all $t$.  In particular the vector space of orthogonal Jacobi fields has dimension $2$.
The set of orthogonal Jacobi fields can be identified with the subbundle of $TT^1M$ whose fiber over $v$ is $v^\perp\oplus v^\perp\subset T_pM\oplus T_pM$. This fiber is the orthogonal complement in $T_vT^1M\simeq T_pM\oplus v^\perp$ of the subspace spanned by the vector field that generates the geodesic flow.

Given $v\in T^1M$ denote
\[
	F_v\eqdef 
	\{\xi\in T_vT^1M\colon\llangle\xi,V(v)\rrangle=0\}
	.
\]
By the above, we have $F_v=\Psi^{-1}(v^\perp\oplus v^\perp)$ and $dg^tF_v=F_{g^t(v)}$ for all $v\in T^1M$ and $t\in\bR$. 

In Section~\ref{ss:almost} we will further study the norm of the linearized flow, where we use the usual definition of the norm of a linear operator  considering the family of scalar products induced by the Sasaki metric:
\[
	\lVert dg^t_v\rVert
	\eqdef \max_{\xi\in T_vT^1M,\,\,\llangle\,\xi,\xi\,\rrangle=1}
		\llangle dg^t_v(\xi),dg^t_v(\xi)\rrangle^{1/2}.
\]

\subsection{Expansivity}\label{sec:exp}

A continuous flow $\psi_{t}\colon X \longrightarrow  X$ without singular points on a  metric space $(X,d)$ is \emph{expansive}%
\footnote{Observe that, in fact, this definition of expansivity is slightly stronger than in~\cite{BowWal:72}, however it appears naturally in the context of expansive geodesic flows (see, for example,~\cite{Rug:94,Rug:97}). In the setting of the geodesic flow  in this paper, both definitions are equivalent.} 
if there exists $\varepsilon>0$ such that for every $x \in X$ and 
for every $y \in X$ for which there exists a continuous surjective function $\rho\colon \bR \longrightarrow \bR$ with $\rho(0)=0$ satisfying
\[
 	d(\psi_{t}(x), \psi_{\rho(t)}(y)) \leq \varepsilon
\]	
for every $t \in {\mathbb R}$ we have $\psi_{t(y)}(x)=y$ for some $\lvert t(y) \rvert <\varepsilon$. We call such $\varepsilon$ an \emph{expansivity constant}.

Let us argue that so-called flat strips in our setting are the only obstruction to expansivity. If the geodesic flow fails to be expansive then there exist two complete geodesics which stay in bounded Hausdorff distance.  It hence follows from the Flat Strip Theorem \cite{EbeONe:73} that they are edges of a flat strip, that is, an isometrically and totally embedded copy of $[0,r]\times\bR$, for some $r>0$.  
This implies that a connected compact  Riemannian surface of nonpositive curvature and negative Euler characteristic manifold has no flat strip if and only if the geodesic flow is expansive.

\subsection{Invariant cone fields}

We introduce the following cone
\begin{equation}\label{def:conecC}
	\cC_v\eqdef \{\xi\in F_v\colon \langle d\pi_v(\xi),C_v(\xi)\rangle_{\pi v}\ge0\}
	\subset F_v.
\end{equation}
The cone field $\{\cC_v\}_{v\in T^1M}$ is forward invariant under the geodesic flow (though not everywhere strictly invariant), that is, it satisfies
\begin{equation}\label{e.cone}
	dg^t_v(\cC_v)\subset \cC_{g^t(v)}
	\quad\text{ for all }\quad v\in T^1M,\quad  t\ge 0.
\end{equation}
 Indeed, we can express an orthogonal Jacobi field as $J(t)=y(t)e_\perp(t)$, where $y$ is a scalar function and $e_\perp$ is a continuous unit vector field along $t\mapsto \gamma(t)$ that is orthogonal to $\dot\gamma$. Then the Jacobi equation reads as
\begin{equation}\label{eq:Jacobi-simp}
	y''(t)+K(\gamma(t))y(t)=0,
\end{equation}
where $K$ denotes the \emph{sectional curvature}. It follows that $u\eqdef y'/y$ satisfies the \emph{Riccati equation}
\[
	u'(t)+u(t)^2+K(\gamma(t))=0.
\]
Now observe that the cone $\cC_v$ is determined by solutions $y$ of~\eqref{eq:Jacobi-simp} satisfying $y(0)y'(0)\ge0$. Note that then
\[
	\big(y\,y'\big)'(t)=-K(\gamma_v(t))\,y^2(t)+\big(y'(t)\big)^2\ge0.
\]
Thus, $t\mapsto y(t)\,y'(t)$ is non-decreasing as $t$ increases and~\eqref{e.cone} follows.

In fact, one can prove a stronger result. Given $k>0$, for every $v\in T^1M$ consider the following cone
\begin{equation}\label{def:conecCk}
	\cC^k_v
	\eqdef \big\{\xi\in \cC_v\colon \lVert d\pi_v(\xi)\rVert_{\pi v} \le k\lVert C_v(\xi)\rVert_{\pi v}\big\}
\end{equation}
and note that $\cC^k_v\backslash\{0\}$ is contained in the interior of $\cC_v$.
By~\cite[Proposition 2.7]{Ebe:73b}, if $k>0$ is such that $K>-k^2$ then any solution of~\eqref{eq:Jacobi-simp} with $y(0)=0$ satisfies  
\[
	\lvert y'(t)\rvert \le k\coth(kt)\,\lvert y(t)\rvert
\]
for every $t>0$ and hence for every $t\ge 0$ we have
\begin{equation}\label{eview}
	dg_v^t(\cC_v)
	\subset \cC^{k\coth(kt)}_{g^t(v)}
	\subset \cC_{g^t(v)}.
\end{equation}
Hence, to prove the following corollary, it suffices to take $\kappa\eqdef k\coth(\tau k)$.

\begin{corollary}\label{cor:coninv}
	For every $\tau>0$ there exists $\kappa=\kappa(\tau)>0$ such that for every $v\in T^1M$ and every $t\ge \tau$ we have
\[
	dg^t_v(\cC_v)
	\subset \cC^\kappa_{g^t(v)}.
\]	
\end{corollary}

\subsection{Invariant vector bundles}\label{sec:invlinfie}

(Orthogonal) un-/stable Jacobi fields provide a convenient geometric way of describing the vector
 bundles that by Oseledec theorem correspond to nonpositive (nonnegative) Lyapunov exponents of the geodesic flow on the unit tangent bundle. 
As curvature is nonpositive, the function $t\mapsto \lVert J(t)\rVert$ is convex~\cite[IV, Lemma 2.3]{Bal:95}.
 A Jacobi field $J$ along a geodesic is called {\em center stable} (resp.~{\em center unstable}) if $\lVert J(t)\rVert$ is bounded for all $t \geq 0$ (resp.  bounded for all $t \leq 0$). 
Using the isomorphism $\Psi$ described in Section~\ref{ss:jf}, let $J^\cs$ (resp. $J^\cu$) denote the $2$-dimensional subspace of center stable (resp. center unstable) Jacobi fields  and introduce the subspaces
\begin{equation}\label{eq:defFcuFcs}
	F^\cs_v\eqdef\{\xi\in T_vT^1M\colon J_\xi\in J^\cs\},\quad
	F^\cu_v\eqdef\{\xi\in T_vT^1M\colon J_\xi\in J^\cu\}.	
\end{equation}
A Jacobi field $J$ along a geodesic is called {\em  stable} (resp.~{\em  unstable}) if it is a center stable (resp. center unstable) Jacobi field and if it is orthogonal. 
Denoting by $J^\s$ (resp. $J^\u$) the stable (resp. unstable) Jacobi fields we introduce the subspaces
\begin{equation}\label{eq:defFuFs}
	F^\s_v\eqdef\{\xi\in T_vT^1M\colon J_\xi\in J^\s\},\quad
	F^\u_v\eqdef\{\xi\in T_vT^1M\colon J_\xi\in J^\u\}.	
\end{equation}
Each of these subspace is in $v^\perp\oplus v^\perp$ and is one-dimensional. 

The vector bundles $F^\ast\colon v\in T^1M\mapsto F^\ast_v\subset T_vT^1M$, $\ast\in\{\s,\cs,\u,\cu\}$ obtained in this way are invariant, that is, for every $v\in T^1M$ and every $t\in\bR$ we have
\[
	dg^t_vF^\ast_v=F^\ast_{g^t(v)},
	\quad\ast\in\{\s,\cs,\u,\cu\},
\]
 and continuous (but rarely have higher regularity).  
The subspace $F^\s_v$ (resp. $F^\u_v$) coincides with the space of vectors $\xi\in v^\perp\oplus v^\perp\subset T_vT^1M$ such that $\lVert dg^t_v(\xi)\rVert$ is uniformly bounded for all $t\ge0$ (resp. bounded for all $t\le0$). 
Note that $F^\u_v=F^\s_{-v}$.

A nonzero vector  $\xi$ belongs to $F^\cs_v\cap F^\cu_v$  if and only if $t\mapsto \|J_\xi(t)\|$ is constant (as a convex bounded map), or in other words if and only if  the function $t\mapsto \lVert dg^t_v(\xi)\rVert$ is constant. One says in this case that $J_\xi$ is a {\em parallel Jacobi field} along $\gamma_v$.
Since we assume that  $M$ is a surface, $F^\cs_v\cap F^\cu_v$ is nontrivial
 if and only if $F^\cs_v=F^\cu_v$, that is if and only if the sectional curvature along $\gamma_v$ is everywhere zero.
In general, both subbundles will have nonzero intersection at some vectors $v\in T^1M$.
 In fact, the geodesic flow is Anosov if and only if the intersection is zero at \emph{any} vector~\cite{Ebe:73b}. This is the case if, for example, the sectional curvature is strictly negative.
 
Invariance of the distribution $F^\u$ together with Corollary~\ref{cor:coninv} imply the following result. 

\begin{corollary}\label{cor:invconFu}
	Given $k>0$ satisfying $K>-k^2$, for every $v\in T^1M$ we have $F^\u_v\in \cC^k_v$.
\end{corollary}

\begin{remark}
Even though we will not further use this fact, observe that a further consequence is that $F^\u_v=F^\s_v=:\bR\xi$ if and only if $C_v(\xi)=0$ (the stable/unstable Jacob field has only a horizontal component).
\end{remark}

The \emph{rank} of a vector $v\in T^1M$ is the codimension of the space $F^\s_v\oplus F^\u_v$ in $T_vT^1M$. The rank is $1$ on the \emph{regular set $\cR$} and $2$ on the \emph{higher rank set $\cH$}. The set $\cR$ is open and invariant. It is also dense~\cite{Bal:82}. A vector $v\in T^1M$ belongs to $\cR$ if and only if the geodesic $\gamma_v$ passes through a point at which the curvature is negative. The complementary set $\cH$ is closed, invariant, and nowhere dense and $v\in\cH$ if and only if the curvature at $\gamma_v$ is constant $0$.

Orthogonal Jacobi fields provide a \emph{continuous} vector bundle that defines the following \emph{continuous} potential which is of great importance for many thermodynamic properties of the flow.
Consider the so-called \emph{geometric potential} defined by
\begin{equation}\label{phidef}
	\varphi^{(\u)}(v)
	\eqdef -\frac{d}{dt}\log\,\lVert dg^t|_{F^\u_v}\rVert|_{t=0} 
	= -\lim_{t\to0}\frac1t\log\,\lVert dg^t|_{F^\u_v}\rVert,
\end{equation}
which is well-defined and depends differentiably on $F^\u_v$ and hence continuously on $v$. 
Analogously, we also consider the potential
\[
	\varphi^{(\s)}(v)\eqdef 
	\lim_{t\to0}\frac{1}{t}\log\,\lVert  d{g}^t|_{ F^\s_v}\rVert
	=\lim_{t\to0}\frac{1}{t}\log\,\lVert  d{g}^{-t}|_{ F^\u_v}\rVert
	= -\varphi^{(\u)}(v).
\]

\begin{remark}
Note that $\varphi^{(\u)}$ vanishes on $\cH$ because the norm of any unstable Jacobi field is constant along geodesics  in $\cH$.
\end{remark}

\begin{remark}[basic sets]\label{bbasic}
Recall that a closed $g^t$-invariant set $X\subset T^1M$ is \emph{hyperbolic} if the tangent bundle restricted to $X$ is a Whitney sum of three $dg^t$-invariant subbundles
\[
	T_XT^1M
	=E^\s\oplus E\oplus E^\u,
\]
where $E$ is the one-dimensional bundle tangent to the flow, and if there are positive constants $c$ and $\lambda$ so that
\[
	\lVert dg^t(u)\rVert\le ce^{-\lambda t}\lVert u\rVert,\quad
	\lVert dg^{-t}(w)\rVert\le ce^{-\lambda t}\lVert w\rVert,\quad
\]
for each $u\in E^\s$ and each $w\in E^\u$ and all $t\ge0$. 
A closed $g^t$-invariant set $X\subset T^1M$ is \emph{basic} if $X$ is hyperbolic, the periodic orbits contained in $X$ are dense in $X$, $g^t|_X$ is transitive, and there is an open set $U\supset X$ so that $X=\bigcap_{t\in\bR}g^t(U)$.

For every $v\in X$, $F^\s_v$ (resp. $F^\u_v$)  coincides with the stable subspace $E^\s_v$ (the unstable subspace $E^\u_v$) in the hyperbolic splitting. In particular, since, those subspaces vary continuously, for any sufficiently small neighborhood $U$ of $X$, the angle between the subspaces $F^\s_w$ and $F^\u_w$, $w\in U$, is uniformly bounded away from zero. 

Recall (see for example \cite[Proposition 6.4.16]{KatHas:95}) that for every basic set $X$ there exist positive constants $c$, $\lambda$, and $\delta$ such that if $v\in X$ then for every $w\in T^1M$ satisfying $\rho(g^t(w),g^t(v))<\delta$ for all $t$, $\lvert t\rvert\le T$, then
\[
	\rho(g^t(w),g^t(v))
	<ce^{-\lambda \lvert T-\lvert t\rvert\rvert}\delta.
\]
\end{remark}

\begin{remark}[regularity of $\varphi^{(\u)}$ on basic sets]
In restriction to  each basic set,  by~\cite[Theorem 19.1.6]{KatHas:95} the map $ F^\u_v$ varies H\"older continuously in $v$. Hence the restriction of $\varphi^{(\u)}$ to each basic set is  H\"older continuous. 
\end{remark}

\subsection{Invariant  foliations and local product structure}\label{sec:invfol}

 Although only continuous in general, nevertheless, each of the distributions $F^\ast$ is integrable to a foliation $\sW^\ast$, $\ast\in\{\s,\cs,\u,\cu\}$, respectively. Their description is purely geometric and their existence requires only the completeness of the geodesic flow of a manifold of nonpositive sectional curvature. If the manifold is compact and negatively curved then they coincide with the usual (center) un-/stable manifolds. Just note that $\sW^\s,\sW^\u$ are defined in terms of level sets of Busemann functions  (see~\cite{Ebe:73b} and~\cite[Section VI C]{Ebe:01} for details). 
 
 Relative to the Sasaki metric, the bundles $F^\s$ and $F^\u$ are both orthogonal to the vector field $V$ tangent to the geodesic flow.  Denote by $F^0$ the subbundle tangent to $V$. The following facts are well known (see~\cite{Ebe:73b}).
 
 \begin{lemma}\label{lem:invfol}
 	For any $v\in T^1M$ we have 
\begin{itemize}
\item[(i)] $F^{\cs}_v=F^\s_v\oplus F^0_v$,  $F^\cu_v=F^\u_v\oplus F^0_v$, 
\item[(ii)] $\sW^\cs_v=\bigcup_{t\in\bR}g^t\sW^\s_v$ and  $\sW^\cu_v=\bigcup_{t\in\bR}g^t\sW^\u_v$,
\item[(iii)] $g^t\sW^\ast_v=\sW^\ast_{g^t(v)}$, $\ast\in\{\s,\cs,\u,\cu\}$, for every $t\in\bR$.
\end{itemize}
\end{lemma}

Note that the intersection of the leaves $\sW^\s_v$ and $\sW^\u_v$ in $v$ is transversal if, and only if, the sectional curvature at $v$ is nonzero; otherwise both submanifolds are tangential at this vector. Nevertheless, under the hypothesis that the geodesic flow is expansive, the intersection is precisely $\{v\}$, that is, the manifolds $\sW^\s_v$ and $\sW^\u_v$ intersect topologically transversely at $v$. Indeed, otherwise the intersection would be contained in a flat strip, see Section~\ref{sec:exp}.

By the above we have the following lemma which is just a reformulation of~\cite[Lemma 3.1]{Bal:95} (recall Section~\ref{sec:exp} and compare~\cite[Lemma 4.5]{CouSch:14}).

\begin{lemma}[Local product structure]\label{lem:locpro}
	The geodesic flow admits a \emph{local product structure}, that is, for every $\varepsilon>0$ there exists $\delta>0$ such that for every $w_1,w_2\in T^1M$ satisfying $\rho(w_1,w_2)\le\delta$ there exists a point $[w_1,w_2]\in T^1M$ and a real number  $t$, $\lvert t\rvert\le\varepsilon$ such that
\[
	[w_1,w_2]\in\sW^\s_\varepsilon(g^t(w_1))\cap\sW^\u_\varepsilon(w_2),
\]
where by $\sW^\ast_\varepsilon(u)$ we denote the set of points $w\in\sW^\ast_u$ satisfying $\rho(w,u)\le\varepsilon$, $\ast=\s,\u$ (compare Figure~\ref{fig:1}). 
\end{lemma}
\begin{figure}
\begin{overpic}[scale=.49 
  ]{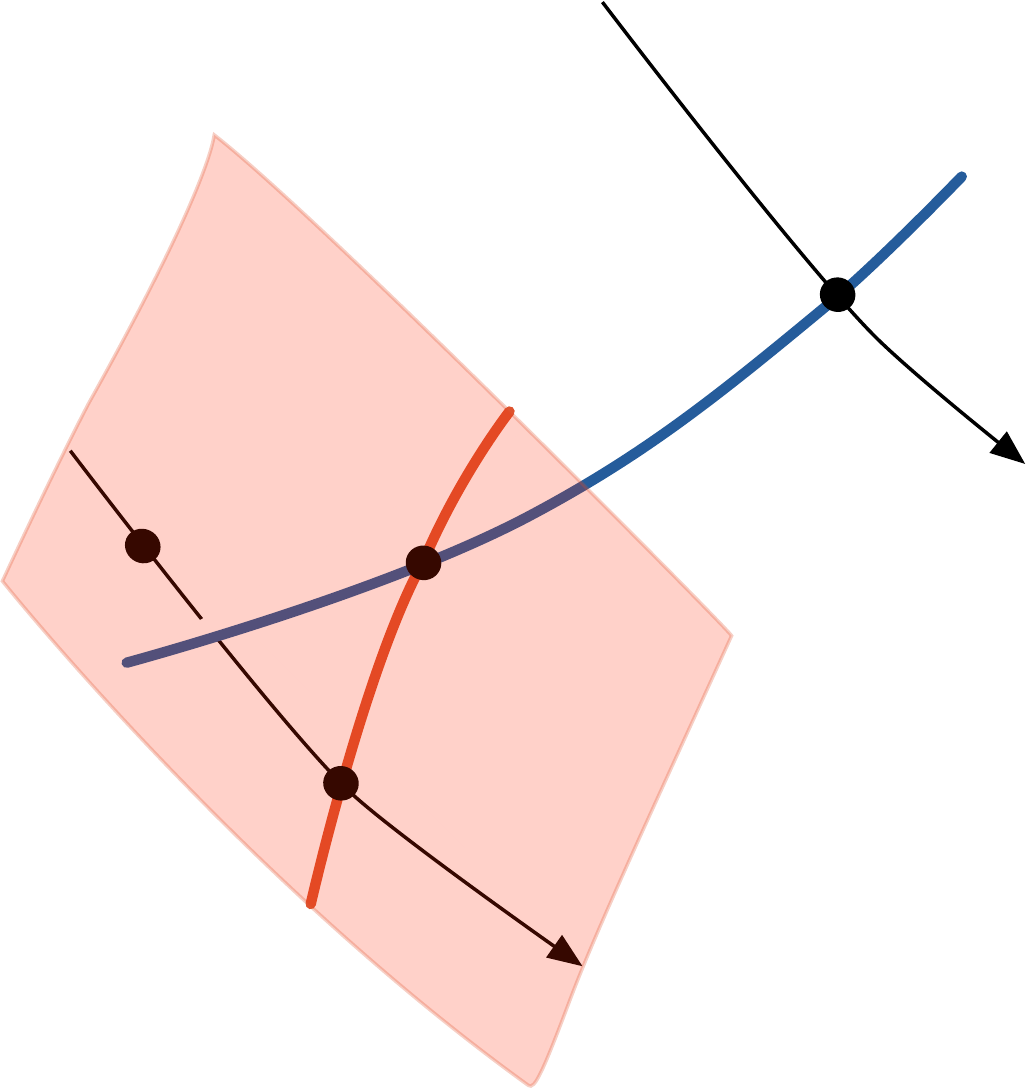}
  	\put(41,45){\tiny $[w_1,w_2]$}
  	\put(80,72){\tiny $w_2$}
  	\put(5,49){\tiny $w_1$}
  	\put(12,27){\tiny $g^t(w_1)$}
  	\put(50,0){\tiny $\textcolor{red}{\sW^\cs_{w_1}}$}
  	\put(45,65){\tiny $\textcolor{red}{\sW^\s_{g^t(w_1)}}$}
  	\put(90,87){\tiny $\textcolor{blue}{\sW^\u_{w_2}}$}
 \end{overpic}
\caption{local product structure}
\label{fig:1}
\end{figure}

\begin{corollary}\label{cor:unstmani}
	Given $\varepsilon>0$ an expansivity constant and $\delta=\delta(\varepsilon/3)>0$ as in Lemma~\ref{lem:locpro} sufficiently small, there exists $\delta>0$ such that given any $v\in T^1M$, for every $w\in T^1M$ satisfying $\rho(g^t(w),g^t(v))\le\delta$ for every $t\le0$ we have $w\in\sW^\cu_v$.
\end{corollary}

\begin{proof}
	Let $\varepsilon>0$ an expansivity constant.
	Let $\delta=\delta(\varepsilon/3)$ be as in Lemma~\ref{lem:locpro} and also assume that $\delta<\varepsilon/3$. 
	First note that for every $u\in\sW^\u_v$ for every $t\le0$ we have $\rho(g^t(u),g^t(v))\le\rho(u,v)$ and for every $w\in\sW^\s_v$ for every $t\ge0$ we have $\rho(g^t(w),g^t(v))\le\rho(w,v)$. 

	Given $v\in T^1M$ and $w\in T^1M$ satisfying $\rho(g^t(w),g^t(v))\le\delta$ for every $t\le0$, by Lemma~\ref{lem:locpro} applied to $w_1=w$ and $w_2=v$, there exists $u\eqdef[w,v]$ and $\tau$, $\lvert \tau\rvert\le\varepsilon/3$, such that $u\in\sW^\s_{\varepsilon/3}(w')\cap\sW^\u_{\varepsilon/3}(v)$, where $w'=g^\tau(w)$. Hence, using the above and the hypothesis, for every $t\ge0$ we have
\[
	\rho(g^t(u),g^t(w'))
	\le \rho(u,w')
	\le \frac\varepsilon3,
\]	
while for every $t\le0$ we have
\[\begin{split}
	\rho(g^t(u),g^t(w'))
	&\le \rho(g^t(u),g^t(v)) + \rho(g^t(v),g^t(w)) + \rho(g^t(w),g^t(w'))\\
	&\le \rho(u,v) + \delta + \rho(w,w') 
	\le \frac\varepsilon3 + \frac\varepsilon3 + \frac\varepsilon3
	=\varepsilon.
\end{split}\]	 
This implies that $\sup_{t\in\bR}\rho(g^t(u),g^t(w'))\le\varepsilon$. By expansivity, hence we can conclude $w'=g^s(u)$ for some $s$ and hence $w=g^{s-\tau}(u)\in\sW^\cu_v$ as claimed.
\end{proof}

\subsection{Lyapunov exponents}
The Lyapunov exponents of the geodesic flow are well defined for all {\em Lyapunov regular} vectors. The set of Lyapunov regular vectors is of full measure with respect to   any invariant probability measure (see for example the appendix of \cite{KatHas:95}). 
 Since we assume that $M$ is a surface, for a {\em Lyapunov (forward) regular} vector $v$ there exists at most one positive Lyapunov exponent $\chi(v)$. Moreover, classical computations give
\begin{equation}\label{eq:Lyapuu}
\chi(v)
 = \lim_{T \to \infty}  \frac1T\log\,\lVert dg^T|_{F^\u_v}\rVert
 = - \lim_{T \to \infty}  \frac1T \int_0^T \varphi^{(\u)}(g^t(v))\,dt.
\end{equation}
Analogously for a Lyapunov backward regular vector.

\begin{remark}
The above observations imply that $\chi(v)=0$ if $\xi\in F^\u_v\cap F^\s_v$, $\xi\ne0$. On the other hand, if $F^\u_v\cap F^\s_v$ is trivial and $\xi\in F^\u_v\setminus F^\s_v$ then $\lVert J_\xi(t)\rVert$ is unbounded but could grow only  sub-exponentially and hence we could have $\chi(v)=0$. If $\chi(v)>0$ then $F^\u_v\cap F^\s_v$ is trivial.
\end{remark}

Given a $G$-invariant Borel probability measure $\mu$, let
\begin{equation}\label{deflyame}
	\chi(\mu)\eqdef\int\chi(v) \,d\mu(v).
\end{equation}

\subsection{Almost multiplicative sequences}\label{ss:almost}

We now derive an almost multiplicative properties of the norm and the inverse norm for the differential $dg^t$. 

We start with some elementary observations. Given a linear invertible transformation between two Banach spaces $L\colon E\to F$, let
\[
	 \lVert L\rVert
	 \eqdef \max_{v, \lVert v\rVert=1}\lVert Lv\rVert 
	 \quad\text{ and }\quad
	[L]
	\eqdef
	\min_{v, \lVert v\rVert=1}\lVert Lv\rVert 
	= \lVert L^{-1}\rVert^{-1}
\]
denote the \emph{norm} and the  \emph{inverse norm} of $L$, respectively. 
A \emph{cone $K$} in a  finite-dimensional Banach space $E$ is a nonempty, convex, closed subset such that $a\xi\in K$ for all $a>0$ and $\xi\in K$, and that $K\cap -K=\{0\}$.

Consider the vector bundle $\{F_v\}_{v\in T^1M}\subset TT^1M$ and the cone field $\{\cC_v\}_{v\in T^1M}$ defined in~\eqref{def:conecC}. Given $\kappa>0$, consider the cone field $\{\cC^\kappa_v\}_{v\in T^1M}$ defined in~\eqref{def:conecCk}. Recall that
\[
	\cC^\kappa_v\subset\cC_v\subset F_v
\]
for all $v\in T^1M$. Recall also that (for any $v$) we have
\[
	\cC^k_v\setminus\{0\}\subset\interior(\cC_v).
\]
We say that the tangent map $dg^t$ \emph{satisfies a cone condition with the cone fields $\cC$ and $\cC^\kappa$} if for every $v\in T^1M$ we have
\[
	dg^t_v(\cC_v)\subset\cC^\kappa_{g^t(v)}.
\]
The following lemma is now an immediate consequence.
\begin{lemma}\label{lem:justin}
	For every $\tau>0$ there exists a  constant $C=C(\tau)\ge1$  such that for every $v\in T^1M$ and for every $s,t$ satisfying $\lvert s\rvert,\lvert t\rvert\ge\tau$ we have 
\[
	C^{-1}\lVert dg^s|_{F_v}\rVert\lVert dg^t|_{F_{g^s(v)}}\rVert
	\le \lVert dg^{s+t}|_{F_v}\rVert
	\le \lVert dg^s|_{F_v}\rVert\lVert dg^t|_{F_{g^s(v)}}\rVert
\]
and
\[
	[ dg^s|_{F_v}][ dg^t|_{F_{g^s(v)}}]
	\le [ dg^{s+t}|_{F_v}]
	\le C[ dg^s|_{F_v}][ dg^t|_{F_{g^s(v)}}].
\]
\end{lemma}

\begin{proof}
To show the first assertion, apply for example arguments in~\cite[Section 2]{FenShm:14}.  For that note that, by Corollary~\ref{cor:coninv}, for every $\tau>0$ there exists $\kappa=\kappa(\tau)>0$ such that for every $t\ge\tau$ the tangent map $dg^t$ satisfies a cone condition with $\cC$ and $\cC^\kappa$. Note that in the horizontal and vertical coordinates provided by the isomorphism $\Psi$, both cone fields are constant, that is, their representations in those coordinates  do not depend on $v$. 

To conclude the second assertion, it suffices to observe that
\[
	[dg^t_w]
	= [dg^{-t}_{-w}]
	= \lVert(dg^{-t}_{-w})^{-1}\rVert^{-1}
	= \lVert dg^t_{g^{-t}(-w)}\rVert^{-1}
\]
and to apply the first assertion.
\end{proof}

For further reference, note that, since the geodesic flow preserves the Liouville measure, for every $v\in \bR$ and every $t\in\bR$ we have
\begin{equation}\label{eq:premeas}
	\lVert dg^t_v\rVert[dg^t_v]=1.
\end{equation}

To determine de norm $\lVert dg_v^t\rVert=\sup_{\lVert \xi\rVert=1}\lVert dg^t_v\xi\rVert$, recall that $\lVert dg_v^t\rVert$ is the largest singular value of the linear operator $dg_v^t\colon T_vT^1M\to T_{g^t(v)}T^1M$. Associated to the singular values there is a system of orthogonal subspaces $E_i\subset T_vT^1M$ such that for a unit length vector $\eta_i\in E_i$, $\lVert dg_v^t\eta_i\rVert$ equals the corresponding singular value.
If $\eta\in T_vT^1M$ is any unit length vector, consider its orthogonal decomposition $\eta=\xi+V(v)$, where $\xi\in T_vT^1M$ is orthogonal to $V(v)$. Because $dg^t_v\xi$ and $dg^t_vV(v)=V(g^t(v))$ are orthogonal and $V$ on $T^1M$ is a vector field of unit length (and hence $\lVert dg_v^tV(v)\rVert=\lVert V(g^t(v))\rVert=1$), $\bR V(v)$ corresponds to one subspace $E_i$ (with singular value $1$). Thus, we obtain
\begin{equation}\label{eq:equinorms}
	\lVert dg^t|_{F_v}\rVert
	=\lVert dg^t_v\rVert.
\end{equation}

We also obtain the following comparison relative to the bundle $F^\u$ defined in~\eqref{eq:defFuFs}. 

\begin{corollary}\label{cor:dis}
	For every $\tau>0$ there exists a constant $C=C(\tau)\ge1$ such that for every $v\in T^1M$ and for every $t\ge\tau$ we have 
	\[
	 \lVert dg^t_v\rVert 
	 \le C \lVert dg^t|_{F^\u_v}\rVert  
	 \quad\text{ and }\quad
	\lVert dg^{-t}_v\rVert 
	\le C\lVert dg^{-t}|_{F^\s_v}\rVert .
	\]	
\end{corollary}

\begin{proof}
If $k>0$ is such that $K>-k^2$ then by Corollary~\ref{cor:invconFu} we have $F^\u\subset \cC^k$.
Note again that, by Corollary~\ref{cor:coninv}, for every $\tau>0$ there exists $\kappa>0$ such that for every $t\ge\tau$ the tangent map $dg^t$ satisfies a cone condition with $\cC$ and $\cC^\kappa$. In fact, by~\eqref{eview} we can chose $\kappa=\kappa(\tau)\eqdef k\coth(k\tau)$.
By arguments in~\cite[Section 2]{FenShm:14}, there exists a constant $C=C(\kappa)\ge1$ such that for any unit vector $\xi\in\cC^\kappa_v$ we have
\[
	C^{-1}\lVert dg^t|_{F_v}\rVert
	\le\lVert dg_v^t(\xi)\rVert.
\]
Taking now $\xi\in F^\u_v$ and also applying~\eqref{eq:equinorms} and \eqref{eq:premeas} we get
\[
	 C^{-1}\lVert dg_v^t\rVert
	= C^{-1}\lVert dg^t|_{F_v}\rVert
	\le \lVert dg^t|_{F^\u_v}\rVert
	\le \lVert dg_v^t\rVert.
\]

The proof for the second assertion is analogous observing $g^t(v)=-g^{-t}(-v)$ and recalling that $F^\u_{-v}=F^\s_v$.
\end{proof}

\begin{proof}[Proof of Proposition~\ref{pro:whams}]
The equalities are immediate consequences of Lemma~\ref{lem:justin} and Corollary~\ref{cor:dis} together with~\eqref{eq:equinorms} and~\eqref{eq:Lyapuu} (we emphasize that we assume in particular the existence of the limits).
\end{proof}

\section{Some thermodynamic preliminaries}\label{sec:prelim}

A principal tool for the proof of our main result is the well developed theory of equilibrium measures of H\"older continuous potentials on basic sets. This concerns in particular information about their coding and precise description of their Gibbs structure. Let us recall some standard concepts.

Given a continuous map $T\colon X\to X$ of a compact metric space $(X,d)$, we denote by $\cM(T)$ the space  of all $T$-invariant Borel probability measures and  endow it with the weak$\ast$ topology.   
Given a continuous function $\varphi\colon X\to\bR$, we denote by $P_{T}(\varphi)$ the \emph{topological pressure} of $\varphi$ (with respect to $T$) and recall  the following \emph{variational principle}
\[
	P_T(\varphi)
	=	\sup_{\mu\in \cM(T)} \left(h(\mu)+\int \varphi\, d\mu\right),
\] 
where $h(\mu)$ denotes the \emph{metric entropy} of $\mu$ (with respect to $T$). A measure attaining the supremum is called an \emph{equilibrium measure} for $\varphi$ (with respect to $T$) (see~\cite{Wal:82} for details). 

In~\cite{BowRue:75} there is given the definition of the \emph{topological pressure of a continuous flow} $G=(g^t)_{t\in\bR}$ on a compact metric space $X$, which we denote by $P_G(\varphi)$, which is equivalent to defining it as the topological pressure of the function $\varphi^1\colon X\to\bR$ defined by $\varphi^1(x)\eqdef \int_0^1\varphi(g^s(x))\,ds$ with respect to the time-$1$ map $g^1$ of the flow,
\[
	P_G(\varphi)
	= P_{g^1}(\varphi^1).
\] 
We denote by $\cM(G)\eqdef\bigcap_{t\in\bR}\cM(g^t)$ the set of all $G$-invariant probability measures on $X$ and by $\cM_{\rm e}(G)$ the subset of all ergodic measures in $\cM(G)$. Analogously to the case of maps, one calls $\mu\in\cM(G)$ an \emph{equilibrium measure} for $\varphi$ (with respect to the flow $G$) if 
\[
	P_G(\varphi)
	= h(\mu)+\int\varphi\,d\mu,
\]
where $h(\mu)$ here denotes the metric entropy of $\mu$ (with respect to $g^1$). Recall also that the latter notation is justified since, denoting by $h(g^t,\mu)$ the metric entropy of $\mu$ with respect to the time-$t$ map, by~\cite{Abr:59} for any $t\in\bR$ we have $h(g^t,\mu)=\lvert t\rvert h(g^1,\mu)$. 

\section{Suspension over an increasing family of SFT's}\label{s:sus}

The aim of this section is to construct an appropriate family of disjoint local cross sections and to model the expansive geodesic flow as a suspension flow.

Let $X\subset T^1M$ be a compact set which is invariant under the flow. Given an interval $I\subset \bR$ and a set $A\subset X$, denote $g^{I}(A)\eqdef\bigcup_{t\in I}g^t(A)$. A set $S\subset X$ is called a \emph{local cross section of time $\zeta>0$ for $G$ and $X$} if $S$ is closed and $S\cap g^{[-\zeta,\zeta]}(x)=\{x\}$ for all $x\in S$ and we let 
\[
	S^{\ast X}\eqdef S\cap{\rm int}\,g^{[-\zeta,\zeta]}(S)
\]	 
(where ${\rm int}$ denotes  the interior in the relative topology induced by $T^1M$ on $X$). We simply write $S^\ast$ if $X=T^1M$. 

\subsection{Local cross sections}\label{sec:loccro}

Given a vector $v\in T^1M$, let us construct a local cross section $D(v)$ containing $v$ (here we follow~\cite{GelRug:}). This section will be foliated by projections of leaves of the foliation $\sW^\u$ (recall Section~\ref{sec:invfol}). 

We first parametrize a neighborhood of a vector $v$ (compare Figure \ref{fig:3}). Given $\varepsilon>0$ sufficiently small let $\delta>0$ be as provided by Lemma~\ref{lem:locpro}. Given $v\in T^1M$, consider the stable leaf $\sW^\s_v$ and the unstable leaf $\sW^\u_v$, each parametrized by arc length. 
 Let
\[
	R_v\colon U\to T^1M,
	\quad U\eqdef \left(-\frac\delta2,\frac\delta2\right)^3\subset\bR^2
\]
be the map with the following properties (compare also Figure~\ref{fig:1}):
\begin{itemize}
\item $R_v(0,0,0)=v$,
\item $t\mapsto R_v(0,0,t)$ is the arc length parametrization of the flow line $t\mapsto g^t(v)$,
\item $r\mapsto R_v(r,0,t)$ is the arc length parametrization of $\sW^\u_{R_v(0,0,t)}$,
\item $s\mapsto R_v(0,s,t)$ is the arc length parametrization of $\sW^\s_{g^t(v)}$.\end{itemize}
\begin{figure}[h]
\begin{overpic}[scale=.48 
  ]{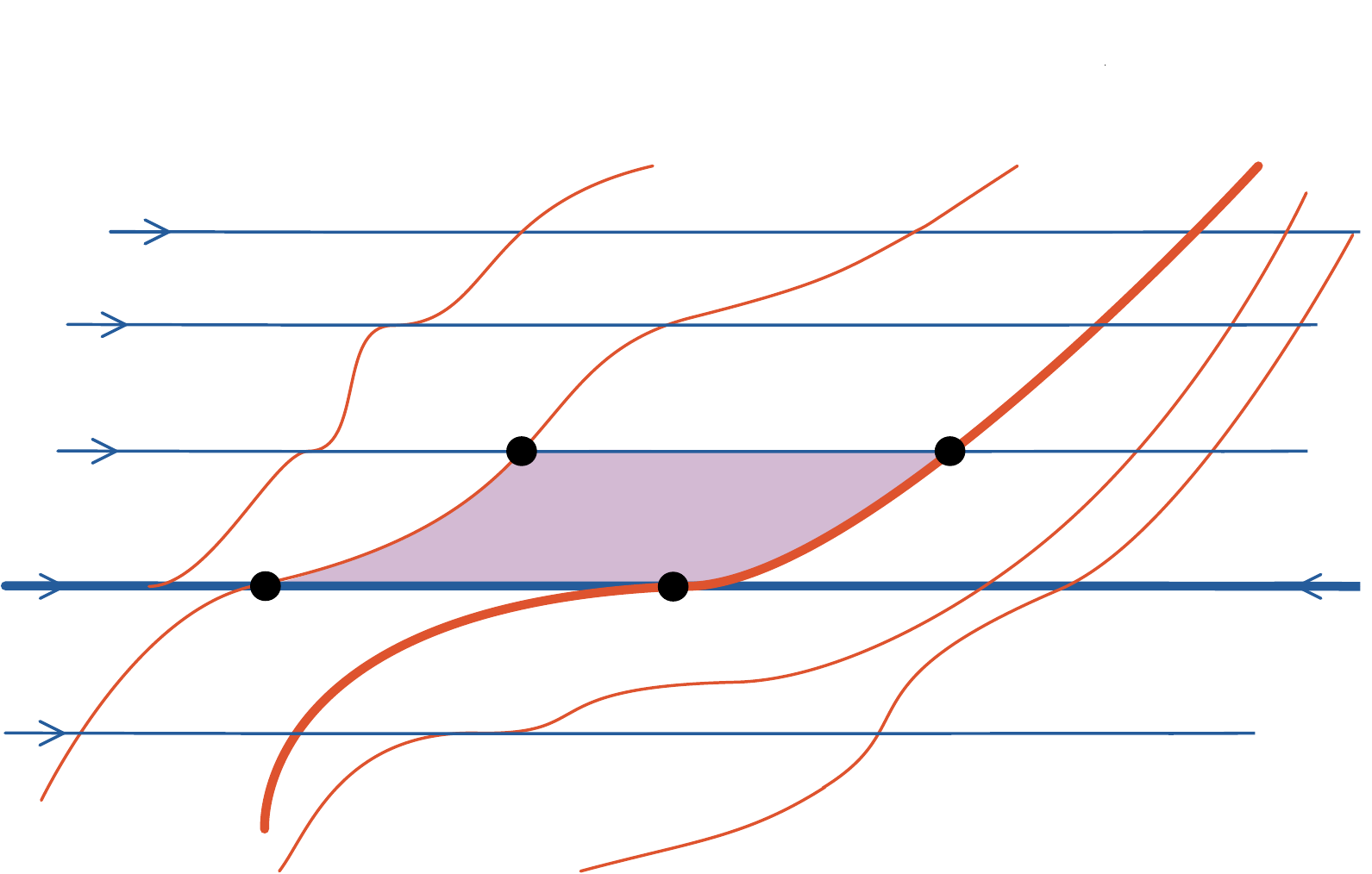}
  	\put(39,33.5){\small $R(r,s,0)$}
      \put(70,27){\small$w_2=R(0,s,0)$}
      \put(-1,34){\small$\sW^\u_{R(r,s,0)}$}
      \put(17,17){\small$w_1=R(r,0,0)$}
      \put(-1,24){\small${\sW^\u_v}$}
      \put(89,54){\small$\sW^\s_v$}
      \put(71,54){\small$\pi\sW^\s_{w_1}$}
      \put(50,17){\small$v=R(0,0,0)$}
\end{overpic}
\begin{overpic}[scale=.48 
  ]{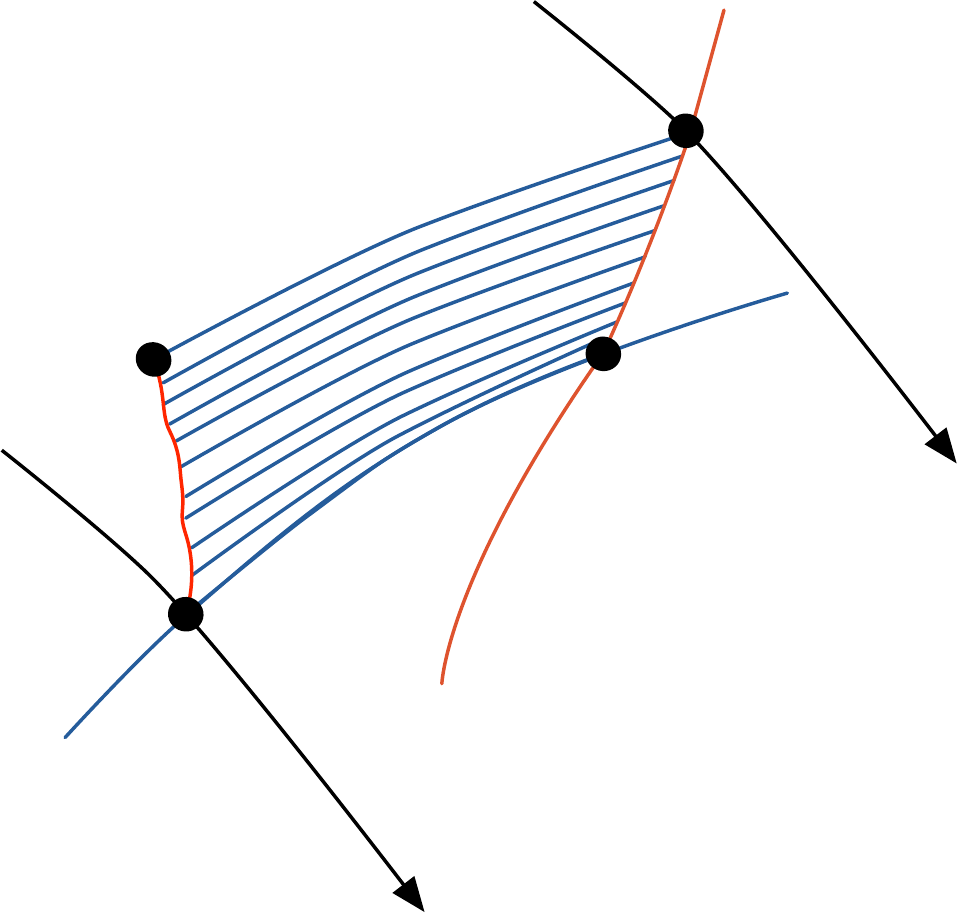}
  	\put(6,67){\small $[w_1,w_2]$}
  	\put(25,32){\small $w_1$}
  	\put(76,83){\small $w_2$}
  	\put(65,53){\small $v$}
  	\put(78.5,99){\small $\sW^\s_v$}
  	\put(10,18){\small $\sW^\u_v$}
\end{overpic}
\caption{Parametrization $R=R_v$ of the local cross section in a neighborhood of a vector $v$. Here $\pi$ denotes the projection of the centre stable leaf onto the cross section given by the local product structure.}
\label{fig:3}
\end{figure}
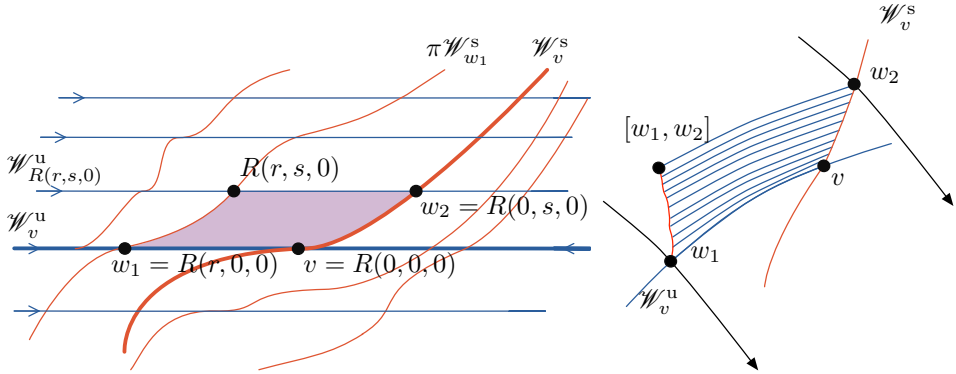
Recall that the bundle $F^\cs=F^\s\oplus V$ is integrable and invariant (Section~\ref{sec:invfol}). 
By the local product structure stated in Lemma~\ref{lem:locpro}, if $w_1=R_v(r,0,t)$ and $w_2=R_v(0,s,t)$ satisfy $\rho(w_1,w_2)\le \delta$ then there exists 
\[
	u
	\eqdef [w_1,w_2]\in\sW^\s_{g^\tau(w_1)}\cap \sW^\u_{w_2}
	= \sW^\cs_{w_1}\cap\sW^\u_{w_2}
\]	 
for some real number $\tau$, $\lvert \tau\rvert\le\varepsilon$, such that $\rho(u,g^\tau(w_1))\le\varepsilon$ and $\rho(u,w_2)\le\varepsilon$. To finish the definition of $R_v$, we let
\begin{itemize}
\item $R_v(r,s,t)\eqdef [R_v(r,0,t),R_v(0,s,t)]$.
\end{itemize}

Recall that $\sW^\u$ is a continuous foliation. Hence, $R_v$ is continuous and injective. By the Invariance of Domain theorem by Brouwer, $R_v(U)$ is open and $R_v$ is a homeomorphism between $U$ and $R_v(U)$.  

We now define 
\[
	D(v)
	\eqdef \left\{R_v(r,s,0)\colon (r,s)\in\left(-\frac\delta2,\frac\delta2\right)^2\right\}.
\]

\begin{remark}\label{rem:sizecharts}
Note that $D(v)$ contains $\sW^\u_{\delta/2}(v)$ and $\sW^\s_{\delta/2}(v)$ and is contained in $R_v(U)$. The map $(r,s)\mapsto R_v(r,s,0)$ defines a parametrization of $D(v)$. Note that $D(v)$ is foliated by Lipschitz curves which are properly contained in the leaves $\sW^\u_\varepsilon(R_v(0,s,0))$ of the unstable foliation.
\end{remark}

Recall that we denoted by $\rho$ the distance on $T^1M$ induced by the Riemannian metric on the surface $M$ and hence on the unit tangent bundle $T^1M$. Given a section $D(v)$, we consider the distance induced by $\rho$ on $D(v)$, for every $w,w'\in D(v)$ let
\[
	\rho_{D(v)}(w,w')
	\eqdef \rho(w,w').
\]

The following result is an immediate consequence of the construction and invariance properties of the unstable foliations (see Lemma~\ref{lem:invfol}  (ii)).

\begin{lemma}\label{lem:locinv}
	The family of cross sections $\{D(v)\}_{v\in T^1M}$ is locally invariant in the sense that for any $v$ and for any $t\in\bR$ the set $g^t(D(v))$  is foliated by leaves of the foliation $\sW^\u$ and contains a subset of the leaf $\sW^\s_{g^t(v)}$.
\end{lemma}

\subsection{Family of disjoint local cross sections}

We will build on constructions in~\cite{BowWal:72}, for completeness we provide full details and adapt them to our needs. Our aim is to consider a \emph{common} family of cross sections which serves at the same time for the flow on $T^1M$ as well as for its restriction to a  compact invariant subset $X\subset T^1M$ (later we will assume $X$ to be a basic set).

If $S\subset X$ is a local cross section of time $\varepsilon$ for $g$, then $g^{(-\varepsilon,\varepsilon)}(S^{\ast X})$ is open for all $\varepsilon>0$ and $g^{[-\varepsilon,\varepsilon]}(S\setminus S^{\ast X})$ is closed and without interior. 
Given a family $\eS=\{S_k\}$ of local cross sections, denote 
\[
	\cup\eS\eqdef \bigcup_kS_k.
\]	 
Note that the flow $(x,t)\mapsto g^t(x)$ maps homeomorphically $S\times[-\zeta,\zeta]$ onto the compact set $g^{[-\zeta,\zeta]}(S)$.

\begin{proposition}\label{pro:comcrosec}
	There is $\zeta>0$ so that the following holds. For every $\hat v\in T^1M$ and every $\alpha\in(0,1)$ there is a finite family $\eS=\{S_k\}$ of pairwise disjoint sets $S_k$ so that:
\begin{itemize}
\item $S_1$ contains $\hat v$ and satisfies $\hat v\in{S_1}^\ast$, that is, $S_1$ contains a neighborhood of $\hat v$ (relative to the induced topology of $T^1M$ on $S_1$).
\item Each $S_k$ is a topological two-manifold which is foliated by unstable leaves,  is a local cross section of time $\zeta$, and has diameter at most $\alpha$. 
\item We have $T^1M=g^{[-\alpha,0]}(\cup\eS)=g^{[0,\alpha]}(\cup\eS)$.
\end{itemize}
	Moreover, if $X\subset T^1M$ is a compact set which is invariant, has topological dimension one, and contains $\hat v$, then $\eS_X\eqdef\{\tilde S_k\}$ with $\tilde S_k\eqdef S_k\cap X$ satisfies 
\[
	X= g^{[-\alpha,0]}(\cup \eS_X)
	=g^{[0,\alpha]}(\cup \eS_X)
\]	 
and we have $S=S^{\ast X}$ for each $S\in\eS_X$.
\end{proposition}

\begin{proof}
We consider the family of local cross sections $\{D(v)\}_{v\in T^1M}$ defined in Section~\ref{sec:loccro}.

For each $v\in T^1M$ we can choose a local cross section $S_v\subset D(v)$ of time $2\zeta_v>0$ so that $v\in {{ S}_{v}}^\ast$. We can assume that $S_v=D(v)\cap B(v,\varepsilon_v)$ for some $\varepsilon_v>0$. By compactness, there exist $\{v_1,\ldots,v_m\}$ so that
\[
	T^1M
	\subset \bigcup_{j=1}^mg^{(-\zeta_{v_j},\zeta_{v_j})}({S_{v_j}}^\ast).
\]
Let
\[
	\zeta
	\eqdef \min_{j=1,\ldots,m}\zeta_{v_j}.
\]
Thus, for every $v\in T^1M$ there is $v_j$ and $r_v\in(-\zeta_{v_j},\zeta_{v_j})$ so that $v\in g^{r_v}({S_{v_j}}^\ast)$ and $T_v\eqdef  g^{r_v}(S_{v_j})$ is a local cross section of time $\zeta_{v_j}\ge 2\zeta$ and $v\in {T_v}^\ast$. Note also that $T_v$ is contained in a topological two-manifold which is foliated by unstable leaves (Lemma~\ref{lem:locinv}).

For $X\subset T^1M$ one-dimensional the set $(T_v\cap X)\times[-2\zeta,2\zeta]\subset X$ is at most one-dimensional. By~\cite{Hur:35} it follows that $T_v\cap X$ is zero-dimensional. Note that $(T_v\cap X)^{\ast X}$ is an open subset of $T_v\cap X$ containing $v$. As $T_v\cap X$ is zero-dimensional, there exists an open-closed neighborhood $U_v$ of $v$ in $T_v\cap X$ which is contained in $(T_v\cap X)^{\ast X}$. Then the set $g^{(-2\zeta,2\zeta)}(U_v)$ is an open subset of $g^{(-2\zeta,2\zeta)}((T_v\cap X)^{\ast X})$ and, as the latter is an open subset of $X$, $g^{(-2\zeta,2\zeta)}(U_v)$ is also an open subset of $X$. Thus, $U_v$ is a local cross section of time $2\zeta$ which contains $v$ and satisfies ${U_v}^{\ast X}=U_v$. 

Given $\alpha\in(0,1)$, let $\varepsilon>0$ satisfying $\varepsilon\le \min\{\alpha/4,\zeta\}$ and $\diam g^r(A)<\alpha$ whenever $\lvert r\rvert\le \varepsilon$ and $\diam A<\varepsilon$. For each $v\in T^1M$ let $V_v\subset {T_v}^\ast$ be a closed set containing a neighborhood of $v$ in $T_v$ with $\diam V_v<\varepsilon$. Then $V_v$ is a local cross section of time $2\zeta$ and $v\in {V_v}^\ast$. 
As in the hypotheses of the proposition, given any vector $\widehat v\in T^1M$ , by  compactness of $T^1M$ there exist $\{v_2,\ldots,v_L\}$ so that
\[
	T^1M
	\subset g^{[-\varepsilon,\varepsilon]}(V_{\widehat v})
		\cup \bigcup_{i=2}^Lg^{[-\varepsilon,\varepsilon]}(V_{v_i}).
\]  

Note that the sets $V_{v_i}$ and the set $V_{\hat v}$ are not necessarily pairwise disjoint.
We now construct finite families $\eS_k$ of pairwise disjoint local cross sections recursively. 

Let $\eS_0\eqdef\{V_{\widehat v}\}$ and put $S_1=V_{\widehat v}$. Hence, in particular, we have $\widehat v\in{V_{\widehat v}}^\ast$, satisfying one of the assertions of the proposition. 

Given $k\in\{2,\ldots,L\}$, suppose that the family $\eS_{k-1}$ is already constructed. Consider the section $V_{v_k}$ and note that for each $v\in V_{v_k}$, $g^{[-\varepsilon,\varepsilon]}(v)\cap \cup\eS_{k-1}$ is a finite set of points (since $\eS_{k-1}$ is a finite family of local cross sections) and (by continuity of the flow and since $\cup\eS_{k-1}$ is a closed set)  there is an open interval $I_v\subset(-\varepsilon,\varepsilon)$ and a closed set $Y_v\subset V_{v_k}$ containing a neighborhood of $v$ in $V_{v_k}$ such that $g^{I_v}(Y_v)\cap\cup\eS_{k-1}=\emptyset$. Since $V_{v_k}$ is compact, there exists $\{v_1,\ldots,v_o\}$ such that 
\[
	V_{v_k}
	\subset \bigcup_{\ell=1}^oY_{v_\ell}.
\]
Let $\tau_k\eqdef\min_\ell\lvert I_\ell\rvert$. Since $V_{v_k}$ is a local cross section of time $\varepsilon\ge\tau_k$ and hence the flow defines locally a homeomorphism on $V_{v_k}\times[-\tau_k,\tau_k]$, we can choose distinct numbers $s_1,\ldots,s_o\in(0,\tau_k)$ such that $g^{[-s_\ell,s_\ell]}(Y_{v_\ell})$ are pairwise disjoint. Let $\zeta_k\eqdef\min_\ell s_\ell$.
Choose numbers $r_1,\ldots,r_o$ such that $r_1+s_1\in I_{v_1},\ldots,r_o+s_o\in I_{v_o}$ and let
\[
	\eS_k
	\eqdef\eS_{k-1}\cup\{g^{r_1+s_1}(Y_{v_1}),\ldots,g^{r_o+s_o}(Y_{v_o})\}. 
\]
Note that this is a family of pairwise disjoint sets each of them contained in topological two-manifold such that each is a local cross section of time $\zeta_k$.

Finally, let $\eS\eqdef \eS_L$ and $\zeta\eqdef\min_k\zeta_k$. We have $T^1M=g^{[-2\varepsilon,2\varepsilon]}(\cup\eS)$. Moreover, for every $v\in T^1M$ we have $g^{2\varepsilon}(x)\in g^{[-2\varepsilon,2\varepsilon]}(\cup\eS)$ and hence $v\in g^{[-4\varepsilon,0]}(\cup\eS)\subset g^{[-\alpha,0]}(\cup\eS)$.

To obtain pairwise disjoint local cross sections for $X$ we can proceed as above considering $V_v\subset U_v$ instead of $V_v\subset T_v$ and obtain the family $\eS_X$ with the claimed properties. Since $X$ is assumed to be one-dimensional, it follows that $(S_k\cap X)^{\ast X}=S_k\cap X$. 
\end{proof}

\begin{remark}
	Note that if $\alpha$ is sufficiently small (depending on the neighborhood which, according to Lemma~\ref{lem:locpro}, permits a parametrization using the local product structure) then for every $S_k\in\eS$ there exists $v_k\in T^1M$ such that $S_k\subset D(v_k)$, where $D(\cdot)$ is the family of local cross sections defined in Section~\ref{sec:loccro}. 
\end{remark}

\subsection{Symbolic coding of the flow}

We continue to follow very closely~\cite[Section 5]{BowWal:72}.
Let $\zeta\in(0,1)$ be as in Proposition~\ref{pro:comcrosec} and choose $\alpha<\zeta$ and let $\eS=\{S_k\}_{k=1}^L$ be a family of pairwise disjoint local cross sections as in Proposition~\ref{pro:comcrosec}.
Let
\begin{equation}\label{e.lozin}
	W\eqdef
	\Big\{ v\in T^1M\colon 
		g^t(v)\cap\big(\bigcup_kS_k\setminus S_k^\ast\big)
		=\emptyset\text{ for all }t\Big\}.
\end{equation}
Note that, by definition, we have $g^t(W)=W$ for all $t$.

For every vector $v\in W\cap\cup\eS$ let $(t_j(v))_{j\in\bZ}$
\[
\ldots<t_{-1}(v)<t_0(v)=0<t_1(v)<\ldots
\]
be the doubly infinite sequence of all \emph{transition times} $t\in\bR$ so that $g^t(v)\in\cup\eS$. In the following we will simply write $t(v)\eqdef t_1(v)$. Since $\eS$ is a family of pairwise disjoint compact local cross sections of time $\zeta>\alpha$, there exists a number $\beta\in(0,\alpha)$ such that the difference of the transition times between consecutive sections must be within the interval $[\beta,\zeta]$, that is, $\beta\le t_{j+1}(v)-t_j(v)\le \zeta$ for all $j$.
Define the \emph{transition map} between sections 
\[	
	\cT\colon W\cap\cup\eS \to W\cap\cup\eS
		\quad\text{ by }\quad\cT(v)\eqdef 
		g^{t(v)}(v).
\]

Consider the shift space $\Sigma_L\eqdef\{1,\ldots,L\}^\bZ$ on which we introduce the metric $d_a(\underline i,\underline i')\eqdef a^{-m}$, $a>1$, where $m=m(\underline i,\underline i')$ is the largest integer such that $i_j=i_j'$ for all $\lvert j\rvert\le m$.  Define
\begin{equation}\label{e.lozinQ}
\begin{split}	
	Q\colon W\cap\cup\eS& \to\Sigma_L\\
	Q(v) \eqdef (\ldots i_{-1}i_0i_1\ldots)\in\Sigma_L
	\,&\text{ so that }\, g^{t_j(v)}(v)\in S_{i_j}\text{ for all }j\in\bZ.
\end{split}
\end{equation}
Consider the usual left shift $\sigma\colon \Sigma_L\to\Sigma_L$. Note that 
\[
	\sigma(Q(W\cap\cup\eS))= Q(W\cap\cup\eS)
\]	 
and that
\begin{equation}\label{e.ddlo}
	\Sigma_Q\eqdef \overline{Q(W\cap\cup\eS)}\subset \Sigma_L
\end{equation}
 is compact and satisfies $\sigma(\Sigma_Q)=\Sigma_Q$. Observe, however, that $\sigma\colon\Sigma_Q\to\Sigma_Q$ is in general not a subshift of finite type.

Let $\Psi$ be the set of all pairs $(v,\underline i)\in\cup\eS\times\Sigma_L$ for which there exists a sequence $\underline t(v,\underline i)=(t_j)_{j\in\bZ}$ such that $t_0=0$, $t_{j+1}-t_j\in[\beta,\zeta]$ for all $j\in\bZ$, and $g^{t_j}(v)\in S_{i_j}$ for all $j\in\bZ$. By~\cite[Lemma 8]{BowWal:72}, $\Psi\subset\cup\eS\times\Sigma_L$ is closed, for each $(v,\underline i)\in\Psi$ the so defined two-sided sequence of real numbers $\underline t(v,\underline i)$ is unique, and 
\[
	\underline t\colon\Psi\to\bR^\bZ,
	\quad
	(v,\underline i)\mapsto \underline t(v,\underline i)=(t_j)_{j\in\bZ}
\]
is continuous.
Denote by $\pi_1\colon \cup\eS\times\Sigma_L\to\cup\eS$ and $\pi_2\colon\cup\eS\times\Sigma_L\to\Sigma_L$ the natural projections to the first and second coordinate, respectively.
Let $\tau\colon \Sigma_Q\to (0,\infty)$ be defined by
\begin{equation}\label{eq:deftau}
	\tau(\underline i)
	\eqdef (\underline t((\pi_2|_\Psi)^{-1}(\underline i)))_1
	= (\underline t(v,\underline i))_1.
\end{equation}
By the above, $\tau$ is continuous. Let
\[
	\tau_{\rm min}
	\eqdef \min_{\underline i\in\Sigma_Q}\tau(\underline i)>0.
\]
Moreover, by~\cite[Lemma 9]{BowWal:72}, there is a continuous map 
\begin{equation}\label{eq:defpi}
	\Pi\colon \Sigma_Q\to \cup\eS
\end{equation}
so that  $\Pi Q(v)=v$ for every $v\in W\cap\cup\eS$, $(\Pi(\underline i),\underline i)\in \Psi$ for $\underline i\in\Sigma_Q$, and $\Pi$ is one-to-one over $W\cap\cup\eS$. More precisely (see~\cite[Proof of Lemma 9]{BowWal:72}), $\Pi$ is defined by $\Pi(\underline i)\eqdef \pi_1((\pi_2|_\Psi)^{-1}(\underline i))$.

Given $m\ge0$ and a finite sequence $(j_{-m}\ldots j_m)\subset\{1,\ldots,L\}^{2m+1}$, we consider the usual cylinder set $[j_{-m}\ldots j_m]\eqdef\{\underline i\colon i_k=j_k\text{ for all }k=-m,\ldots,m\}$. Correspondingly, let
\[
	S_{j_{-m}\ldots j_m}
	\eqdef \Pi([j_{-m}\ldots j_m])
\]
denote a ``cylinder local cross section''. 

\begin{remark}[Common transition times]\label{rem:invtratime}
 	Note that by choice of the family of local cross sections, we have $S_{i_{-m}\ldots i_m}\subset S_{i_0}\subset D(v_{i_0})$. Moreover, recalling Lemma \ref{lem:locinv}, we observe that  the transition time from the cylinder local cross section $S_{i_{-m}\ldots i_m}\subset S_{i_0}$ to the local cross section $S_{i_m}$ is constant on $S_{i_{-m}\ldots i_m}$ and, in particular,  $g^{t_m(v_{i_0})}(S_{i_{-m}\ldots i_m})\subset S_{i_m}$. Analogously,  the transition time from  $S_{i_{-m}\ldots i_m}$ to $S_{i_{-m}}$ is constant on $S_{i_{-m}\ldots i_m}$ and  $g^{t_{-m}(v_{i_0})}(S_{i_{-m}\ldots i_m})\subset S_{i_{-m}}$. 
\end{remark}

The following is an immediate consequence of expansivity of the flow $G$.

\begin{lemma}\label{lem:expansive}
	For every $\underline i\in\Sigma_Q$ we have 
\[
	\lim_{m\to\infty}\diam S_{i_{-m}\ldots i_m}=0.
\]	
\end{lemma}

\subsection{Symbolic coding of the flow on a basic set}\label{sec:symcodflobas}

Consider any compact invariant one-dimensional set $X\subset T^1M$ and let $\eS_X=\{\tilde S_k\}=\{S_k\cap X\}$ as in Proposition~\ref{pro:comcrosec}.
Analogously to~\eqref{e.lozin}, let
\[
	W_X\eqdef
	\Big\{ v\in X\colon 
		g^t(v)\cap\big(\bigcup_k\tilde S_k\setminus {\tilde S_k}^{\ast X}\big)
		=\emptyset\text{ for all }t\Big\}.
\]
Note that $g^t(W_X)=W_X$ for all $t$ and that, by the second claim of Proposition~\ref{pro:comcrosec}, 
\[
	W_X=X.
\]
Analogously to~\eqref{e.lozinQ}, define 
\[
	Q_X\colon W_X\cap\cup\eS_X \to\Sigma_L
\]
and, analogously to~\eqref{e.ddlo}, define
\[
	\Sigma^X
	\eqdef\overline{Q_X\big(W_X\cap\cup\eS_X\big)}
	\subset\Sigma_L.
\]
Note that
\[
	\sigma(\Sigma^X)=\Sigma^X\subset\Sigma_Q.
\]
In particular, $\underline t(v,\underline i)$ is well-defined for each $(v,\underline i)\in\Psi\cap(\cup\eS_X\times\Sigma^X)$ and likewise $\Pi(\underline i)$ is well-defined for each $\underline i\in\Sigma^X$.

By~\cite[Lemma 9]{BowWal:72}, $\Pi\colon\Sigma_Q\to\cup\eS$ is continuous  and a bijection between $\Sigma^X$ and its image $\Pi(\Sigma^X)$ and hence $\Pi|_{\Sigma^X}$ is a homeomorphism between $\Sigma^X$ and its image. 

Recall that a compact invariant set $\tilde\Sigma\subset\Sigma_L$ is a \emph{subshift of finite type} (\emph{SFT}) if it is determined by specifying blocks of finite length, that is, there exist $N\ge 1$ and a finite set of (``forbidden'') $N$-blocks $\cF=\{(i_1\ldots i_N)\}\subset\{1,\ldots,L\}^N\}$ such that 
\[
	\tilde\Sigma=\tilde\Sigma(\cF)
	=\{\underline i\in\Sigma_L
		\colon (i_k\ldots i_{k+N-1})\not\in\cF
		\text{ for all }k\in\bZ\},
\]	 
and we call $N$ the \emph{admissible length}.
For the following classical result see also  \cite{Bow:72b}.

\begin{lemma}
	Let $X\subset T^1M$ be a basic one-dimensional set.
	Then $\Sigma^X$ is a SFT with the admissible length depending on $X$ (which can be chosen arbitrarily large).
\end{lemma}

\subsection{Suspension flow}\label{subsec.symflo}

Consider the compact metric space $(\Sigma_Q,d)$ and the homeomorphism $\sigma|_{\Sigma_Q}\colon \Sigma_Q\to \Sigma_Q$ (which we continue to denote simply by $\sigma$). 
Consider the \emph{height function} defined in~\eqref{eq:deftau} and recall that $\tau$ is continuous. The \emph{suspension of $\sigma|_{\Sigma_Q}$ under $\tau$} is the flow $F=(f^t)_t$ on the space
\[
	\Sigma_Q(\sigma,\tau)
	\eqdef
	{\Sigma_Q\times[0,\alpha]}/\sim,
\]
where $\sim$ is the identification of $(\underline i,\tau(\underline i)+s)$ with $(\sigma(\underline i),s)$ for all $s\ge0$ such that $\tau(\underline i)+s\le\alpha$, defined by 
\[
	f^r(\underline i,s)\eqdef (\underline i,r+s)
	\quad\text{ for any }\quad
	0\le r+s<\tau(\underline i).
\]	 
The space $\Sigma_Q(\sigma,\tau)$ is a compact metrizable metric space (see~\cite[Section 2]{BowWal:72} for a metric).
Recall the definition of $\Pi$ in~\eqref{eq:defpi}. Let $p\colon 	\Sigma_Q(\sigma,\tau)\to M$ be defined by 
\begin{equation}\label{eq:factor}
	p(\underline i,s)\eqdef g^s(\Pi(\underline i)).
\end{equation}
By~\cite[Theorem 10]{BowWal:72}, the flow $G$ on $T^1M$ is the factor of the suspension flow $F$ on $\Sigma_Q(\sigma,\tau)$, that is, we have
\[
	p\circ f^t
	= g^t\circ p,
\]
with the factor map $p$ being a homeomorphism between  invariant Baire sets.

\begin{remark}[basic sets]\label{rem:bassetsus}
Given a basic sect $X\subset T^1M$ and its associated SFT $\Sigma^X\subset\Sigma_Q$ considered in Section~\ref{sec:symcodflobas}, there is the analogously defined suspension of this SFT under $\tau$ and the corresponding flow on the space 
\[
	\Sigma^X(\sigma,\tau)
	\eqdef \Sigma^X\times[0,\alpha]/\sim.
\]
This suspension flow is simply the restriction of the suspension flow $F$ to $\Sigma^X(\sigma,\tau)$.
By~\cite[Theorem 10]{BowWal:72}, the flow $G$ restricted to $X$ is the factor of the restricted suspension flow $F$ on $\Sigma^X(\sigma,\tau)$ by means of the factor map $p$, with $p$ being a homeomorphism between $\Sigma^X(\sigma,\tau)$ and $X$. Observe that here we use the hypothesis that $X$ is one-dimensional.
By~\cite[Theorem 1]{Bow:72b}, the geodesic flow $G=(g^t)_t$ restricted to any basic set is isomorphic to a hyperbolic symbolic flow, that is, to a suspension of a SFT under a Lipschitz continuous height function. 
\end{remark}

\subsection{Measures and potentials associated to the suspension}
Based on the above, we can now rely on the theory developed in~\cite{BowWal:72,Bow:72b,Bow:73,BowRue:75}.
There is a canonical identification between invariant measures for the suspension flow  and the invariant measures for the subshift in the sense that
for any $\nu\in\cM(\sigma|_{\Sigma_Q})$ and the Lebesgue measure $m$ on $\bR$, the measure $\mu_\nu$ defined by 
\begin{equation}\label{eq:defmeaa}
	\mu_\nu
	\eqdef \frac{1}{(\nu\times m)(\Sigma_Q(\sigma,\tau))}
			(\nu\times m)|_{\Sigma_Q(\sigma,\tau)}
\end{equation}
 is a probability measure on $\Sigma_Q(\sigma,\tau)$ and the identifications take place on a set of measure zero. Moreover, $\mu_\nu$ is invariant under the suspension flow $F=(f^t)_t$ and the map $\nu\mapsto \mu_\nu$ between $\cM(\sigma|_{\Sigma_Q})$ and $\cM(F|_{\Sigma_Q(\sigma,\tau)})$ is one-to-one.
By Abramov's theorem~\cite{Abr:59}, for any $\nu\in\cM(\sigma|_{\Sigma_Q})$ and the corresponding measure $\mu_\nu\in\cM(F|_{\Sigma_Q(\sigma,\tau)})$ we have
\[
	h(F,\mu_\nu)
	=\frac{h(\sigma,\nu)}{\int\tau\,d\nu}.
\]

Note that,  on a basic set $X\subset T^1M$, by the above mentioned factor map $p$ defined in \eqref{eq:factor} between the  associated suspension flow $F$ on $\Sigma^X(\sigma,\tau)$ and the flow $G$ (see Section~\ref{subsec.symflo}), there is also a bijection between $\cM(F|_{\Sigma^X(\sigma,\tau)})$ and $\cM(G|_X)$.
More precisely, recall again that the flow $G|_X$ is the factor $p$ of the restricted suspension flow $F$ on $\Sigma^X(\sigma,\tau)$, with the factor map $p$ being a homeomorphism between $\Sigma^X(\sigma,\tau)$ and $X$ (recall Remark~\ref{rem:bassetsus}). If $\lambda\in\cM(G|_X)$ is the equilibrium measure of some H\"older continuous potential $\phi$ on $X$, then 
\[
	\mu
	=p^\ast\lambda
	\eqdef\lambda\circ p\in\cM(F|_{\Sigma^X(\sigma,\tau)})
\]	 
is the (unique) equilibrium measure for $\phi^\ast\eqdef\phi\circ p$. Note that both measures have the same entropy and that we have
\[
	P_{G|_X}(\phi)
	=h(G,\lambda)+\int\phi\,d\lambda
	=h(F,\mu)+\int\phi^\ast\,d\mu
	=P_{F|_{\Sigma^X(\sigma,\tau)}}(\phi^\ast).
\]
Moreover,  $\mu=\mu_\nu$ as for~\eqref{eq:defmeaa}, where $\nu\in\cM(\Sigma^X)$ is the  equilibrium measure of the potential $\Delta_\phi- P_{G|_X}(\phi)\cdot\tau$, where $\Delta_\phi\colon\Sigma^X\to\bR$ is defined by
\begin{equation}\label{laut}
	\Delta_\phi(\underline i)\eqdef 
	\int_0^{\tau(\underline i)} \phi\big({g}^s(\Pi(\underline i))\big)\,ds 
\end{equation}
 (see~\cite[Proposition 3.1]{BowWal:72}).
Together with Fubini's theorem, we hence obtain
\[
	\int\phi\,d\lambda
	= \int\phi^\ast\,d\mu_\nu
	= \frac{\int\Delta_\phi\,d\nu}{\int\tau\,d\nu}.
\]	

Observe that
\[
	\sum_{k=0}^{m-1}\Delta_\phi(\sigma^k(\underline i)) 
	= \int_0^{ T_{ m}(\underline i)}
		\phi\big({g}^t(\Pi(\underline i))\big)\,dt, 
\]
where $T_m(\underline i)$ denotes the  \emph{forward finite transition times of level $m$} at $\underline i$ defined by
\begin{equation}\label{defTm} 
	T_{m}(\underline i) \eqdef 
	\sum_{k=0}^{m-1}\tau(\sigma^k(\underline i)).
\end{equation}	

\section{Bridging measures on one-sided shift spaces}\label{sec:bridge}

We construct a Borel probability measure $\nu$ on the space $\Sigma_L^+$ that we call a \emph{bridging measure} which has the property that the orbit of a $\nu$-typical point has  -- on a given sequence of finite time-intervals --   ``finite-time entropies'' and ``finite-time averages'' that are very close to the entropy and ergodic averages of the measures $\nu_\ell$ from a given sequence $(\nu_\ell)_\ell$. Here we follow constructions in~\cite{GelRam:09,GelPrzRam:10}.
Such a bridging measure  mimics asymptotically the asymptotic behaviour of $\nu_\ell$ as $\ell\to\infty$.  Note that, in general, this measure is not invariant with respect to the shift $\sigma\colon\Sigma_L^+\to\Sigma_L^+$. 

The constructions and results in this section are completely general and  independent from the remaining sections of this paper. 
Since in the entire section we consider the one-sided shift space $\Sigma_L^+=\{0,\ldots,L\}^\bN$, to simplify notation we skip the symbol ${}^+$ in the notation of subspaces,  sequences, cylinders, etc.

\subsection{Construction of a bridging measure $\nu$ on $\Sigma_Q$}

We make the following hypothesis.
\smallskip

\noindent
\textbf{Hypothesis 1 -- Symbolic part:} 
Let $\Sigma_Q\subset\Sigma_L^+$ be a closed $\sigma$-invariant set and consider an increasing family $(\Sigma^\ell)_{\ell\ge1}$ of subshifts of finite type
\[
\Sigma^1\subset\Sigma^2\subset\ldots\subset\Sigma_Q\subset\Sigma_L^+.
\]
Consider a sequence $(\nu_\ell)_\ell$ of Borel probability measures each being the equilibrium measure of some H\"older continuous potential $\phi_\ell\colon\Sigma^\ell\to\bR$ with respect to $\sigma|_{\Sigma^\ell}$.
Without loss of generality%
\footnote{Observe that the claimed result does, in fact, not depend on the actual potential but only on the associated equilibrium measure $\nu_\ell$. Hence, without changing the equilibrium measure $\nu_\ell$, we can assume that the potential $\phi_\ell$  satisfies 
\begin{equation}\label{e.nanana}
	P_{\sigma|_{\Sigma^\ell}}(\phi_\ell)
	= h(\nu_\ell)+\int\phi_\ell\,d\nu_\ell = 0.
\end{equation}
Indeed, otherwise we can replace $\phi_\ell$ by the potential $\widetilde\phi_\ell=\phi_\ell-P_{\sigma|_{\Sigma^\ell}}(\phi_\ell)$ and observe that the latter has the same equilibrium measure such as the former and that its topological pressure is zero.}%
, we can assume that 
\[
	P_{\sigma|\Sigma^\ell}(\phi_\ell)=0. 
\]
In particular, $\nu_\ell$ is a Gibbs measure (see~\cite[Chapter 4]{Bow:75}) and hence gives positive measure to any finite-level cylinder intersecting $\Sigma^\ell$.

Let $\Delta\colon\Sigma_Q\to\bR$ be a continuous function.
Let $\underline a\in\Sigma^1$. Let $m_1$ some positive integer.

\begin{figure}
\begin{overpic}[scale=.48 
  ]{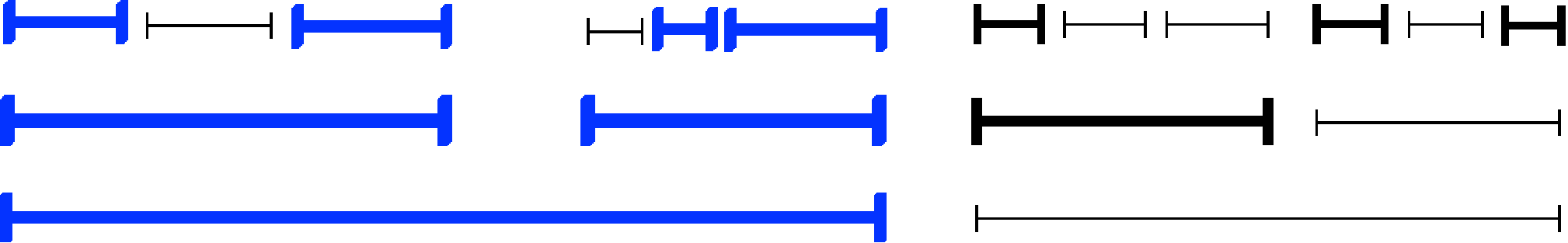}
        \put(105,1){\tiny$\Sigma_{m_1}^1$}
        \put(105,7){\tiny$\Sigma_{m_2}^2$}
        \put(105,13){\tiny$\Sigma_{m_3}^3$}
\end{overpic}
\caption{Schematic construction of $\nu$: $m_\ell$-cylinders which intersect $\Sigma^\ell$ (bold), cylinders on which $\nu$ is distributed (bold blue), $\ell=1,2,3$}
\label{fig:2}
\end{figure}

Let $(m_\ell)_{\ell\ge2}$ be a (sufficiently rapidly, as specified below) increasing sequence of positive integers. 
We will define a Borel probability measure $\nu$ on the Borel $\sigma$-algebra of $\Sigma_L^+$. Recall that this $\sigma$-algebra is generated by (one-sided) cylinders $[i_0\ldots i_{m-1}]$. 
To start, consider the cylinder $[a^{m_1}]=[a_0\ldots a_{m_1-1}]\subset\Sigma_L^+$ of length $m_1$ and define 
\[
	\nu([a^{m_1}])\eqdef 1
	\left(= \frac{\nu_1([a^{m_1}])}{\nu_1([a^{m_1}])}\right)
\]	
(note that since $\nu_1$ is an equilibrium measure and hence has the Gibbs property and $[a^{m_1}]$ intersects $\Sigma^1$, $\nu_1$ gives positive measure to the cylinder $[a^{m_1}]$).
Given $\ell\ge1$, assume now that the measure $\nu$ was already defined on cylinders of length $m_\ell$. We sub-distribute the measure $\nu$ on the sub-cylinders of length $m_{\ell+1}$ which intersect $\Sigma^{\ell+1}$ as follows: Given any finite sequence $i^{m_\ell}=(i_0\ldots i_{m_\ell-1})$ of length $m_{\ell}$ such that $[i^{m_\ell}]$ has positive measure $\nu$, let
\[
\nu([i^{m_\ell}j^{m_{\ell+1}-m_\ell}])
\eqdef N_{\ell+1}(i^{m_\ell})
        \, \nu([i^{m_\ell}])
        \, \nu_{\ell+1}([j^{m_{\ell+1}-m_\ell}]),
\]
where $N_{\ell+1}(i^{m_\ell})$ is the normalizing constant given by
\[
N_{\ell+1}(i^{m_\ell}) \eqdef
\left[
\sum
\nu_{\ell+1}([j^{m_{\ell+1}-m_\ell}])\right]^{-1}
\]
with summation taken over all cylinders $j^{m_{\ell+1}-m_\ell}$ so that $[i^{m_\ell}j^{m_{\ell+1}-m_\ell}]\cap\Sigma^{\ell+1}\ne\emptyset$.
For $m$ with $m_\ell<m<m_{\ell+1}$ let
\[
	\nu([i^{m}]) 
	\eqdef
 	\sum \nu([i^{m}j^{m_{\ell+1}-m}])
\]
with the analogous summation taken (compare Figure~\ref{fig:2}).
We extend the measure $\nu$ arbitrarily to the Borel $\sigma$-algebra of $\Sigma_L^+$. We will call the measure $\nu$ a \emph{bridging measure with respect to $((\Sigma^\ell)_\ell, (\nu_\ell)_\ell, \Delta,[a^{m_1}])$} and to a choice of a sufficiently rapidly increasing sequence $(m_\ell)_\ell$. 

\begin{remark}
Note that the choice of the sequence $(m_\ell)_\ell$ will be done according to certain properties of Birkhoff averages of $\Delta$. Hence, the bridging measure will also depend on $\Delta$ although this dependence is not yet apparent.
\end{remark}

Notice that $\nu$ in general is not $\sigma$-invariant.
Since $\Sigma^\ell\subset\Sigma_Q$ for each $\ell$, the measure $\nu$ is supported on $\Sigma_Q$ only. More precisely, it is supported on $[a^{m_1}]\cap\Sigma_Q$ only.

Denote
\[
	h_\ell
	\eqdef h(\sigma,\nu_\ell),\quad
	\Delta_\ell
	\eqdef \int\Delta\,d\nu_\ell.
\]
We define the \emph{forward finite-time Birkhoff averages of $\Delta$ of level $m$} at $\underline i$ by 
\[
	L_m(\underline i) \eqdef 
	\frac 1 m \sum_{k=0}^{m-1}\Delta(\sigma^k(\underline i))
\]
Recall the definition of the forward finite transition time of level $m$ at $\underline i$, $T_m(\underline i)$, in~\eqref{defTm}. 
Given the Borel probability measure $\nu$, we define its \emph{finite-time local entropy at level $m$} at $\underline i$ by
\[ 
	H_m(\nu,\underline i) 
	\eqdef - \frac{ 1}{m} \log \nu([i_0\ldots i_{m-1}]) .
\]
We define the \emph{distortion} of $\Delta$ near $\Sigma^\ell$  at level $m$ by
\[
	\dist_{\ell,m}\Delta
	\eqdef \max\left\{m\lvert L_{ m}(\underline j) - L_{ m}(\underline i)\rvert 	
			\colon
			\underline i\in\Sigma^\ell,\underline j\in[i_0\ldots i_{m-1}]\right\}.
\]
Note that we compare the distortion for orbits which stay \emph{close to} $\Sigma^\ell$ (that is, we do not necessarily have $\underline j\in\Sigma^\ell$). By continuity of $\Delta$, for each $\ell\ge1$ there exists a sequence $(\rho^\ell_m)_m$ of positive numbers converging to zero such that 
\begin{equation}\label{distorrrrtion}
	\dist_{\ell,m}\Delta
	\le {m\rho^\ell_m}.
\end{equation} 

The following result is in~\cite{GelRam:09}. For completeness, we sketch (parts of) its proof.

\begin{proposition}\label{p.pgfaweiss}
	Assume Hypothesis 1. There exist a choice of a sequence $(m_\ell)_\ell$ of sufficiently fast growing positive integers and a set $K\subset[a^{m_1}]\subset\Sigma_L^+$ such that for the bridging  measure $\nu$ with respect to $((\Sigma^\ell)_\ell,(\nu_\ell)_\ell,\Delta,[ a^{m_1}])$ the restriction of $\nu$ to $K$ is a Borel probability measure on $\Sigma_Q$ for which there exists a full measure set $\widehat K\subset K$ so that for every $\underline i\in\widehat K$ we have
\begin{itemize}
\item[(1)] $\displaystyle\liminf_{m\to\infty}H_m(\underline i)=\liminf_{m\to\infty} h_m$, 
	$\displaystyle\limsup_{m\to\infty}H_m(\underline i)=\limsup_{m\to\infty} h_m$, 
\item[(2)] $\displaystyle\liminf_{m\to\infty}L_m(\underline i)=\liminf_{m\to\infty} \Delta_m$, 
	$\displaystyle\limsup_{m\to\infty}L_m(\underline i)=\limsup_{m\to\infty} \Delta_m$,
\item[(3)] $\displaystyle\liminf_{m\to\infty}\frac{H_m(\underline i)}{L_m(\underline i)}
	=\liminf_{m\to\infty}\frac{h_m}{\Delta_m}$.	 	
\end{itemize}
Moreover, each of these  limits is uniform in $\underline i\in\widehat K$.
\end{proposition}

\begin{proof}
For the cylinder $[a^{m_1}]=[a_0\ldots a_{m_1-1}]$, for a given sequence $(m_\ell)_\ell$ of natural numbers, consider the subsets
\[\begin{split}
	K_{\rm H}(\varepsilon, \ell)
	&\eqdef 
	 \Big\{ \underline i \in[a^{m_1}]\colon\\
	 &\phantom{\eqdef\Big\{\underline i}
	 	\Big\lvert H_m(\nu,\underline i) - \Big(\frac{m_\ell}{m}H_{m_\ell}(\nu,\underline i) 
	+\frac{(m-m_\ell)}{m} h_{\ell+1}\Big)\Big\rvert
	\le \varepsilon  \lvert h_{\ell+1}-h_\ell\rvert \\
	&\phantom{-------------------}
	 \text{ for }m_\ell<m\le m_{\ell+1}\Big\}\\
	 K_{\rm L}(\varepsilon, \ell)
	&\eqdef \Big\{ \underline i \in[a^{m_1}]\colon\\
	&\phantom{\eqdef\Big\{\underline i}
	 \Big\lvert  L_m(\underline i) - \Big(\frac{m_\ell}{m} L_{m_\ell}(\underline i) +
	  \frac{m-m_\ell}{m} \Delta_{\ell+1}\Big)\Big\rvert
	\le \varepsilon \, \lvert \Delta_{\ell+1}-\Delta_\ell\rvert \\
	&\phantom{-------------------}
	\text{ for  }m_\ell<m\le m_{\ell+1}\Big\}.
\end{split}\]

Observe that the finite-time local entropy at each $\underline i\in K_{\rm H}(\varepsilon,\ell)$ at level $m$ for $m_\ell<m\le m_{\ell+1}$ is roughly equal to the convex combination of its finite-level entropy $H_{m_\ell}(\nu,\underline i)$ and the value $h_{\ell+1}$. More precisely,  we have
\[
	\left\lvert
		\log\frac{\nu([i_0\ldots i_{m_{\ell+1}-1}])}
				{\nu([i_0\ldots i_{m_\ell-1}])}
		+ (m_{\ell+1}-m_\ell)h_{\ell+1}
	\right\rvert
	\le 	m_{\ell+1}\varepsilon	 \lvert h_{\ell+1}-h_\ell\rvert.
\]
This observation yields to the following large deviation result which we state without proof (compare~\cite[Proposition 7]{GelRam:09}).

\begin{lemma} \label{cor1}
	For every $\varepsilon>0$ and every $\delta\in(0,1)$ there exists a number $M_{\rm H}=M_{\rm H}(\varepsilon,\delta,\nu_{\ell+1})\ge1$ such that for any choice of $m_\ell\ge M_{\rm H}$ we have 
	\[
	\nu(K_{\rm H}(\varepsilon,\ell))> 1- \delta.
	\]
\end{lemma}

The case of finite-time Lyapunov exponents is easier. Indeed, for any $m$ with $m_\ell<m\le m_{\ell+1}$ the finite-time Birkhoff averages are convex combinations.
\begin{equation} \label{eqn:lups}
	L_m(\underline i) 
	= \frac {m-m_\ell} m L_{m-m_\ell}(\sigma^{m_\ell}(\underline i)) 
		+ \frac {m_\ell} m L_{m_\ell}(\underline i)	\notag.
\end{equation}
For every $\underline i\in K_{\rm L}(\varepsilon, \ell)$ we have
\[
	\left\lvert
		\sum_{k=m_\ell}^{m_{\ell+1}-1}\Delta(\sigma^k(\underline i))
		- (m_{\ell+1}-m_\ell)\Delta_{\ell+1}
	\right\rvert
	\le 	m_{\ell+1}\varepsilon	 \lvert \Delta_{\ell+1}-\Delta_\ell\rvert.
\]
In correspondence to Lemma~\ref{cor1}, we have the following large deviation result (see~\cite[Proposition 8]{GelRam:09}).

\begin{lemma} \label{cor2}
	For every $\varepsilon>0$ and every $\delta\in(0,1)$ there exists an integer $M_{\rm L}=M_{\rm L}(\varepsilon,\delta,\phi_{\ell+1},\nu_{\ell+1})\ge1$ such that for any choice of $m_\ell\ge M_{\rm L}$ we have 
	\[
	\nu(K_{\rm L}(\varepsilon,\ell)		)> 1- \delta.
	\]
\end{lemma}

We make now a first specification of the sequence $(m_\ell)_\ell$.

\begin{lemma}\label{lem:LD}
	For every sequence $(\varepsilon_\ell)_\ell$ decaying to $0$, there exists a choice of sequence $(m_\ell)_\ell$ such that $\nu$-almost every $\underline i\in \Sigma_Q$ is contained in $K_{\rm H}(\varepsilon_\ell,\ell)\cap K_{\rm L}(\varepsilon_\ell,\ell)$ for all except finitely many $\ell$.
\end{lemma}

\begin{proof}
	By Lemmas~\ref{cor1} and~\ref{cor2} we can choose a summable sequence of numbers $(\delta_\ell)_{\ell=1}^\infty$ and a sequence $(M_\ell)_\ell$ such that for any choice $m_\ell\ge M_\ell$ we have
	\[
	\sum_{\ell=1}^\infty
	\nu\Big(\Sigma_Q\setminus 
		\big(K_{\rm H}(\varepsilon_\ell,\ell)\cap K_{\rm L}(\varepsilon_\ell,\ell) 
		\big) \Big) <
		2\sum_{\ell=1}^\infty\delta_\ell
		<\infty.
	\] 
	The Borel-Cantelli lemma now implies that the set of points $\underline i\in \Sigma_Q$ that are contained in infinitely many $K_{\rm H}(\varepsilon_\ell,\ell)\cap K_{\rm L}(\varepsilon_\ell,\ell)$ 
	 is zero.
\end{proof}	

In the rest of this proof fix some sequence $(\varepsilon_\ell)_\ell$ decaying to $0$ and consider a sequence $(m_\ell)_\ell$ as provided by Lemma~\ref{lem:LD}. Given $j\ge 1$ let
\begin{equation}\label{defKj}
K_j\eqdef \bigcap_{\ell\ge j}\Big( K_{\rm H}(\varepsilon_\ell,\ell)\cap 
		K_{\rm L}(\varepsilon_\ell,\ell) 
		\Big)
		\quad\text{ and  let }\quad
		\widehat K\eqdef \bigcup_{j\ge1}K_j.
\end{equation}
Note that $(K_j)_{j\ge1}$ is an increasing family of sets of points for which the above estimates are satisfied for all $\ell\ge j$. 

By Lemma~\ref{lem:LD} we have $\nu(\widehat K)=1$. 
Hence, there is $j\ge1$  such that $\nu(K_j)>0$.
Given $\underline i\in K_j$, by definition of $K_{\rm H}(\varepsilon_\ell,\ell)$, for all $\ell\ge j$ we have
\[
\Big\lvert H_{m_{\ell+1}}(\nu,\underline i) - \frac{m_\ell}{m_{\ell+1}} H_{m_\ell}(\nu,\underline i) 
	-h_{\ell+1} +\frac{m_\ell}{m_{\ell+1}} h_{\ell+1}\Big\rvert
	\le \varepsilon_\ell \, \lvert h_{\ell+1}-h_\ell\rvert.
\]
This implies 
\[\begin{split}
&\Big\lvert  H_{m_{\ell+1}}(\nu,\underline i) - h_{\ell+1}\Big\rvert \\
&\le \Big\lvert H_{m_{\ell+1}}(\nu,\underline i)
	- \frac{m_\ell}{m_{\ell+1}} H_{m_\ell}(\nu,\underline i)
	- h_{\ell+1}
	+\frac{m_\ell}{m_{\ell+1}} h_{\ell+1}\Big\rvert\,+\\
&\phantom{=}+\,	\Big\lvert
	\frac{m_\ell}{m_{\ell+1}}h_\ell
	-\frac{m_\ell}{m_{\ell+1}} h_{\ell+1}
	- \frac{m_\ell}{m_{\ell+1}}h_\ell
	+\frac{m_\ell}{m_{\ell+1}} H_{m_\ell}(\nu,\underline i)
	\Big\rvert \\
&\le  \Big(\frac{m_\ell}{m_{\ell+1}}+\varepsilon_\ell\Big) \lvert h_{\ell+1}-h_\ell\rvert
	+ \frac{m_\ell}{m_{\ell+1}}\Big\lvert H_{m_\ell}(\nu,\underline i) - h_\ell\Big\rvert.
\end{split}\]
Hence, given $\underline i\in K_j$ for every $\ell\ge j$ we have 
\[\begin{split}
&\left\lvert  \frac{H_{m_{\ell+1}}(\nu,\underline i)}{h_{\ell+1}} - 1 \right\rvert \\
&\le \left(\frac{m_\ell}{m_{\ell+1}}+\varepsilon_\ell\right) 
	\left\lvert 1 - \frac{h_\ell}{h_{\ell+1}}\right\rvert
	+ \frac{m_\ell}{m_{\ell+1}}
	\left\lvert \frac{H_{m_\ell}(\nu,\underline i)}{h_\ell} - 1\right\rvert
	\frac{h_\ell}{h_{\ell+1}}.
\end{split}\]
Clearly $H_{m_1}(\nu,\underline i) \in [H^\s, H^\u]$ for numbers $0<H^\s$, $H^\u<\infty$ independent of $\underline i$ since there are only finitely many cylinders. 
Hence, making a particular choice of the sequences  $\varepsilon_\ell\to 0$ and then $m_\ell\to\infty$, the above estimates imply that the convergence $H_{m_\ell}(\nu,\underline i)/h_\ell \rightarrow 1$ is uniform in $\underline i\in K_j$ as $\ell\to \infty$.
Notice that for every $m_\ell\le m \le m_{\ell+1}$
\[
 h(m)\eqdef\frac{m_\ell}{m} h_\ell + \frac{(m-m_\ell)}{m} h_{\ell+1}
\]
satisfies $h_\ell \le h(m)\le h_{\ell+1}$. Now, recalling that $\underline i\in K_{\rm H}(\varepsilon_\ell,\ell)$ for every $\ell\ge j$, we obtain
\[\begin{split}
	\lvert H_m(\nu,\underline i) - h(m)\rvert
	&= \Big\lvert H_m(\nu,\underline i) 
		- \frac{m_\ell}{m}h_\ell - \frac{(m-m_\ell)}{m}h_{\ell+1}\Big\rvert\\
	&\le \varepsilon_\ell\lvert h_{\ell+1}-h_\ell\rvert 
		+ \frac{m_\ell}{m}	\lvert H_{m_\ell}(\nu,\underline i)-h_\ell\rvert,	
\end{split}\]
which implies 
\[
\liminf_{m\to\infty} H_m(\nu,\underline i) = \liminf_{\ell\to\infty} h_\ell
\quad\text{ and }\quad
\limsup_{m\to\infty} H_m(\nu,\underline i) = \limsup_{\ell\to\infty} h_\ell
\]
uniformly in $\underline i\in K_j$, proving item (1) of the proposition.

Note that $L_{m_1}(\underline i) \in [L^\s, L^\u]$ for some numbers $0<L^\s$,  $L^\u<\infty$ independent of $\underline i$ since $\Delta$ as a continuous function is  bounded on $\Sigma_Q$.
By analogous arguments, from the definition of $K_{\rm L}(\varepsilon_\ell,\ell)$ together with
\[
 \Delta(m)\eqdef\frac{m_\ell}{m} \Delta_\ell + \frac{m-m_\ell}{m} \Delta_{\ell+1}
\]
satisfying $\Delta_\ell\le \Delta(m) \le \Delta_{\ell+1}$, we can conclude that 
\[
	\liminf_{m\to\infty} L_{\pm m}(\underline i) = \liminf_{\ell\to\infty} \Delta_\ell
	\quad\text{ and }\quad 
	\limsup_{m\to\infty} L_{\pm m}(\underline i)  = \limsup_{\ell\to\infty} \Delta_\ell
\]	
uniformly in $\underline i\in K_j$, proving item (2) of the proposition.

As for all $m_\ell<m\le m_{\ell+1}$ we have
\[
\frac{h(m)}{\Delta(m)} \ge \min\Big\{\frac{h_\ell}{\Delta_\ell}, \frac{h_{\ell+1}}{\Delta_{\ell+1}}\Big\},
\]
we finally derive
\[
	\liminf_{m\to\infty} \frac {H_m(\nu,\underline i)} {L_m(\underline i)} 
	= \liminf_{\ell\to\infty} \frac{h_\ell}{\Delta_\ell}
\]	
uniformly in $\underline i\in K_j$.

Finally, consider the restriction of $\nu$ to $K_j$ which is a  Borel probability measure having all properties claimed in the proposition on the set $\widehat K\eqdef K_j$.
\end{proof}

\subsection{The modeled counterpart}
We will suppose that the above symbolic dynamics is the model of some system on a compact metric space. Moreover, we will further specify the choice of the sequence $(m_\ell)_\ell$ (note that this will not alter the already obtained properties of the hence constructed bridging measure, see Remark~\ref{rem:biggersequence}).
\smallskip

\noindent
\textbf{Hypothesis 2 -- modeled counterpart:} 
Consider $(M,\rho)$ a one-dimensional smooth manifold and $\Pi\colon[a^{m_1}]\to M$  a continuous map. Assume that for every $\ell\ge1$ 
\[
	\{\Pi([i_0\ldots i_{m_\ell-1}])\colon\underline i \in\Sigma^\ell\}
\]	 
is a family of pairwise disjoint sets.
 Let $\underline\chi,\overline\chi\colon M\to\bR_{>0}$ be two measurable functions such that 
 \[
 	\underline\chi\circ\Pi=\liminf_{m\to\infty}L_m
	\quad\text{ and }\quad
 	\overline\chi\circ\Pi=\limsup_{m\to\infty}L_m.
 \]
Suppose that there is a sequence of positive numbers $(K_\ell)_\ell$ such that for every $\ell\ge1$ for every $\underline i\in\Sigma^\ell\cap[a^{m_1}]$  and every $m=kN_\ell\ge1$, where $N_\ell$ is the admissible length of the SFT $\Sigma^\ell$ and $k\ge1$, we have
\begin{equation}\label{H2:diameter}
	 \diam_M\big(\Pi([i_0\ldots i_{m-1}])\big)
	 	> K_\ell^{-1}
			e^{-m\rho^\ell_m}
			e^{-m L_{m}(\underline i)},
\end{equation}
where $\diam_M$ denotes the diameter relative to $\rho$.
Without loss of generality, we can also assume that 
\begin{equation}\label{H2:const}
	\lim_{\ell\to\infty}\frac{1}{m_\ell}\log K_\ell=0
	\quad\text{ and }\quad
	\lim_{\ell\to\infty}\rho^\ell_{m_\ell}=0.
\end{equation}

\begin{remark}[Choice of sequence $(m_\ell)_\ell$]\label{rem:biggersequence}
Observe that it is a consequence of its proof that the result of Proposition \ref{p.pgfaweiss} remains true if we replace the sequence $(m_\ell)_\ell$ by a sequence $(\widetilde m_\ell)_\ell$ satisfying $\widetilde m_\ell\ge m_\ell$ for every $\ell$. 
Without loss of generality, we can, for example, assume that for each $\ell$, $m_\ell$ is a multiple of the $\Sigma^\ell$-admissible length of this SFT.
Observe that Hypothesis 2 adds additional assumptions about the choice of the sequence $(m_\ell)_\ell$ which hence do not alter the already obtained properties of the corresponding bridging measure.
\end{remark}

Assuming Hypothesis 2, consider $\nu$ the bridging measure and define  
\[
	\lambda
	\eqdef \Pi_\ast\nu
\]
and call it the \emph{bridging measure  with respect to $((\Sigma^\ell)_\ell,(\nu_\ell)_\ell,\Delta,[a^{m_1}])$, the  choice of $(m_\ell)_\ell$, and the modeled counterpart $\Pi\colon[a^{m_1}]\to M$}.

\begin{remark}\label{rem:promeas}
	By construction of the bridging measure $\nu$, for every $\ell\ge1$ and every $\underline i\in\Sigma^\ell$ we have
\[
	\lambda(\Pi ([i_0\ldots i_{m_\ell-1}]) )
	= \nu( [i_0\ldots i_{m_\ell-1}]) 
	.
\]	
\end{remark}

\begin{proposition}\label{p.pgfaweiss-b}
	Assume Hypotheses 1 and 2. There exists a Borel probability measure $\lambda$ on $M$ and a $\lambda$-full measure set $\widetilde K\subset M$ so that for every $x\in\widetilde K$ we have
\[
	\underline d_\lambda(x)
	\eqdef \liminf_{\varepsilon\to0}\frac{\log \lambda(B(x,\varepsilon))}{\log\varepsilon}
	\ge  \liminf_{\ell\to\infty}\frac{h_\ell}{\Delta_\ell}
\]
and
\[
	\underline\chi(x)
	=\liminf_{\ell\to\infty}\Delta_\ell
	\quad\text{ and }\quad
	\overline\chi(x)
	=\limsup_{\ell\to\infty}\Delta_\ell.
\]
\end{proposition}

\begin{proof}
The main idea is to construct a ``multi-scale dynamically defined Cantor set'' which at each level of construction is covered by projections by $\Pi$ of cylinder sets which intersect the sets  $K_{\rm H}$ and $K_{\rm L}$ and which intersect the projection of $\Sigma^\ell$, where we will make explicit reference to the properties of those sets constructed in the proof of Proposition~\ref{p.pgfaweiss}. 
 On those sets (and hence on some sufficiently small cylinder neighborhoods) we have good control of finite-time entropy and finite-time Lyapunov exponents. This will enable us to determine the local dimension at each point of the generated Cantor set.
 
Fix some sequence $(\varepsilon_\ell)_\ell$ decaying to $0$ and choose a sufficiently rapidly growing sequence $(m_\ell)_\ell$ according to Proposition \ref{p.pgfaweiss} and to Remark~\ref{rem:biggersequence} (we will further specify this sequence below). Given $j\ge1$ let 
\begin{equation}\label{eq:herenew}
	\cK_j
	\eqdef\bigcap_{\ell\ge j} \cC_\ell
	\quad\text{ and let }\quad
	\widehat \cK
	\eqdef \bigcup_{j\ge1} \cK_j,
\end{equation}
 where 
\[ 
	\cC_\ell
	\eqdef\bigcup \Pi( [i_0\ldots i_{m_\ell-1}]),
\]
where the union is taken over all cylinders $[i_0\ldots i_{m_\ell-1}]$ for some sequence $\underline i\in\Sigma^\ell$ which have nonempty intersection with the set $K_{\rm H}(\varepsilon_\ell,\ell)\cap K_{\rm L}(\varepsilon_\ell,\ell)$.	
Note that, by Hypothesis 2 and by Remark \ref{rem:promeas}, for the measure  $\lambda=\Pi_\ast\nu$ we have
\[
	\lambda(\Pi ([i_0\ldots i_{m_\ell-1}]) )
	= \nu( [i_0\ldots i_{m_\ell-1}]) 	.
\]	

By  Proposition~\ref{p.pgfaweiss}, there is a $\nu$-full measure set $\widehat K\subset\Sigma_L^+$
\[
	\lim_{\ell\to\infty}\lvert{H_{m_\ell}(\nu,\cdot)}-{h_\ell}\rvert=0
	\quad\text{ and }\quad
	\lim_{\ell\to\infty}\lvert{L_{m_\ell}(\cdot)}-{\Delta_\ell}\rvert=0.
\]
uniformly in $\widehat K$. Hence, given $\varepsilon>0$ there exists $\ell_0$ such that for any $\underline i\in \widehat K$ for any $\ell\ge \ell_0$ we have
\[
	\frac{-\log \nu([i_0\ldots i_{m_\ell-1}])}{m_\ell h_\ell}
	\in(1-\varepsilon,1+\varepsilon)
	\quad\text{ and }\quad
	\frac{L_{m_\ell}(\underline i)}{\Delta_\ell}
	\in(1-\varepsilon,1+\varepsilon).
\]

Let us now estimate the pointwise dimension of the measure $\lambda$.
Let $x\in\widehat \cK$. 
Note that we can write 
\[
	\{x\}
	\subset \bigcap_{\ell\ge1}\Pi ([i_0\ldots i_{m_\ell-1}]) 
\]
for some appropriate symbolic sequence $\underline i\in\Sigma_Q$. Given $x$, by~\eqref{eq:herenew} there is $j\ge1$ such that $x\in\cK_j$ and hence $x\in\cC_\ell$ for every $\ell\ge j$.
Together with Hypothesis 2, 
\[\begin{split}
	\log\diam_M\Pi([i_0\ldots i_{m_\ell-1}])
	&\ge -\log K_\ell - m_\ell\rho_{m_\ell}^\ell - m_\ell L_{m_\ell}(\underline i)\\
	&\ge -\log K_\ell - m_\ell\rho_{m_\ell}^\ell - m_\ell(1-\varepsilon)\Delta_\ell.
\end{split}\]
 On the other hand, 
\[
	\log\lambda(\Pi([i_0\ldots i_{m_\ell-1}]))
	= \log\nu([i_0\ldots i_{m_\ell-1}])
	\le -m_\ell (1-\varepsilon) h_\ell.
\]
Letting now $r_\ell\eqdef \diam_M\Pi([i_0\ldots i_{m_\ell-1}])$ and using the fact that $M$ is one-dimensional smooth manifold and hence $x$ can be in at most two cylinder sets $\Pi([i_0\ldots i_{m_\ell-1}])$ simultaneously, we obtain
\[
	\lambda(B(x,r_\ell))
	\le 2e^{-m_\ell (1-\varepsilon) h_\ell}
\]
which implies
\[
	\frac{\log\lambda(B(x,r_\ell))}
		{\log r_\ell}
	\ge  \frac{-\frac{\log 2}{m_\ell}+ (1-\varepsilon)h_\ell}
			{\frac{\log K_\ell}{m_\ell} + \rho_{m_\ell}^\ell + (1-\varepsilon)\Delta_\ell}.
\]
Taking the limit $\ell\to\infty$ and using~\eqref{H2:const} and recalling that $\varepsilon$ was arbitrary, we obtain
\[
	\underline d_\lambda(x)
	\ge \liminf_{\ell\to\infty}\frac{h_\ell}{\Delta_\ell}.
\]

Finally, let $\widetilde K\eqdef\Pi(\widehat K)$.
The statement about the lower and upper limits $\underline\chi$ and $\overline\chi$ are immediate by Proposition~\ref{p.pgfaweiss}.
This proves the proposition.
\end{proof}

\section{Proof of Theorem~\ref{jen}}\label{sec:locate}

Any increasing family of basic sub-sets of $T^1M$ provides us immediately with lower bounds for the Hausdorff dimension of $\cL(\alpha)$ for exponents $\alpha$ in the \emph{interior} of the spectrum of possible Lyapunov exponents. Our main concern, however, is to describe the level set $\cL(0)$ for the exponent $\alpha=0$ at the \emph{boundary} of the spectrum of possible exponents. The analogous applies to the other boundary value $\overline\chi\eqdef\max\{\alpha>0\colon\chi(\mu)=\alpha\text{ for some }\mu\in\cM_{\rm e}(G)\}$.

Given $v\in T^1M$, denote by
\begin{equation}\label{eq:defexponents}
	\underline\chi^\pm(v)
	\eqdef \liminf_{t\to\pm\infty}\frac1t\log\,\lVert dg^t|_{F_v}\rVert
\end{equation}
the \emph{lower forward/backward Lyapunov exponent} at $v$ (we assume that the limit exists) and define the \emph{upper forward/backward Lyapunov exponent} at $v$ analogously replacing $\liminf$ by $\limsup$ and denote them by $\overline\chi^\pm(v)$. We study the Hausdorff dimension of level sets $\{v\colon \underline\chi^\pm(v)=\alpha_1,\overline\chi^\pm(v)=\alpha_2\}$.

\begin{proposition}\label{prop:boundary}
	Let $(\Lambda_\ell)_{\ell\ge1}$ be an increasing family of basic sets each with topological dimension one and so that $\bigcup_{\ell\ge1}\Lambda_\ell$ is dense in $T^1M$.
For each $\ell\ge1$ let $\lambda_\ell\in\cM(G|_{\Lambda_\ell})$ be an equilibrium measure for some H\"older continuous potential. 
Let $v\in \bigcup_{\ell\ge1}\Lambda_\ell$ and $\varepsilon\in(0,1)$.
Then
	\begin{multline*}
		\dim_{\rm H}\left(\left\{ w\in \sW^\u_\varepsilon(v)\colon 
		\underline\chi^+(w)=\liminf_{\ell\to\infty}\chi(\lambda_\ell),\,
		\overline\chi^+(w)=\limsup_{\ell\to\infty}\chi(\lambda_\ell)\right\}\right)\\
		\ge \liminf_{\ell\to\infty}\frac{h(G,\lambda_\ell)}{\chi(\lambda_\ell)}.
	\end{multline*}
The analogous statement holds true for the local stable manifold $\sW^\s_v$ and the backward Lyapunov exponents.	
\end{proposition}

Using Proposition \ref{prop:boundary}, we are now prepared to give the proof of Theorem~\ref{jen}. Our approach is to ``fill'' $T^1M$ by a family of hyperbolic subsets, invoking results from~\cite{BurGel:14}. 

By \cite{BurGel:14} there exists an increasing family  of (nontrivial, that is, not only one single periodic orbit) basic sets $(\Lambda_\ell)_\ell$ whose union is dense in $T^1M$.

\begin{proof}[Proof of Theorem~\ref{jen}]
We will only prove the assertion for the local unstable manifolds. The assertion for the local stable manifolds is analogous and follows from the naturally given symmetry between forward/backward trajectories and hence unstable/stable manifolds, see \eqref{eq:forbac}.

Recall the definition of the Lyapunov exponent $\chi$ in~\eqref{deflyame}. Given a compact invariant set $\Lambda\subset T^1M$, denote 
\[
	\underline\chi(\Lambda)\eqdef \inf\{\alpha\ge 0\colon \chi(\mu)=\alpha
	\,\,\text{ for some }\mu\in\cM_{\rm e}(G|_\Lambda)
	\}
	\]
and define $\overline\chi(\Lambda)$ analogously replacing $\inf$ by $\sup$. 
Recall that, by our hypothesis, we have $\cH\ne\emptyset$ which implies $\underline\chi(T^1M)=0$.
By~\cite[Theorem 1.4]{BurGel:14} there exists an increasing family of basic sets $(\Lambda_\ell)_{\ell\ge1}$ which satisfy  
	\[
	\lim_{\ell\to\infty}\underline\chi(\Lambda_\ell)=\underline\chi(T^1M)=0
	\quad\text{ and }\quad
	\lim_{\ell\to\infty}\overline\chi(\Lambda_\ell)=\overline\chi(T^1M).
	\]
Moreover, this family can be chosen such that each $\Lambda_\ell$ is one-dimensional and that $\bigcup_{\ell\ge1}\Lambda_\ell$ is dense in $T^1M$.	

The following concept is fundamental in almost any type of multifractal analysis (see, for example, \cite{PesWei:97}), and instrumental in the present considerations.
Given $\alpha\in\bR$, let us  introduce the \emph{Legendre-Fenchel transform of the pressure function} $q\mapsto P_{G|\Lambda}(q\varphi^{(\u)})$ that is defined by
\[
	\cE_{G|\Lambda}(\alpha)\eqdef 
	\inf_{q\in\bR}\left( P_{G|\Lambda}(q\varphi^{(\u)}) + q\alpha\right).
\]
By~\cite[Section 3.3]{BurGel:14}   for every $\alpha\in(0,\alpha_1]$ we have 
\[
	 \frac{1}{\alpha}\cE_{G|T^1M}(\alpha)=1.
\]	
Together with \cite[Proposition 12]{BurGel:14}, for any $\alpha\in(\underline\chi(T^1M),\overline\chi(T^1M))$ we obtain $\lim_{\ell\to\infty}\cE_{G|\Lambda_\ell}(\alpha)=\cE_{G|T^1M}(\alpha)$. Hence, we can choose a monotonically decreasing sequence of exponents $\alpha_k \to 0$ and for any $k$ we can choose an index $\ell=\ell(k)$ such that $\alpha_k\in (\underline\chi_\ell,\overline\chi_\ell)$ and $\cE_{G|\Lambda_\ell}(\alpha_k)/\alpha_k\ge 1-1/k$. For each $\alpha_k$ there exist a unique number $q_k$ and a unique equilibrium measure $\lambda_k$ for the potential $q_k \varphi^{(\u)}$  (with respect to ${G}|_{\Lambda_{\ell}}$) such that $\chi(\lambda_k)=\alpha_k$.  For each such subsystem we hence have
	\[
	h({G}|_{\Lambda_{\ell}},\lambda_k)
	=P_{{G}|\Lambda_{\ell}}(q_k \varphi^{(\u)})+q_k \alpha_k 
	\ge \min_{q\in\bR} \Big(P_{{G}|\Lambda_{\ell}}(q \varphi^{(\u)})+q\alpha_k \Big)
	= \cE_{G|\Lambda_{\ell}}(\alpha_k).
	\]
Hence, by Proposition~\ref{prop:boundary} for every $v\in T^1M$ and $\varepsilon>0$ we have
	\[
	\dim_{\rm H}\left(\left\{ w\in \sW^\u_\varepsilon(v)\colon 
	\chi(w)=\lim_{k\to\infty}\chi(\lambda_k)=0\right\}\right)
	\ge
	 	\liminf_{k\to\infty}\frac{\cE_{G|\Lambda_{\ell(k)}}(\alpha_k)}{\alpha_k}=1
	.\]
	This proves the theorem.
\end{proof}

In the remainder of this section we will prove Proposition~\ref{prop:boundary}.

\begin{proof}[Proof of Proposition~\ref{prop:boundary}]
Let $(\Lambda_\ell)_\ell$ be an increasing family  of one-dimensional basic sets $(\Lambda_\ell)_\ell$ whose union is dense in $T^1M$. Let $v\in \bigcup_{\ell\ge1}\Lambda_\ell$. Without loss of generality, we can assume that $v\in\Lambda_1$. Let $\varepsilon\in(0,1)$ be an expansivity constant and let $\delta=\delta(\varepsilon/3)$ be as in Lemma~\ref{lem:locpro}.
\smallskip

\noindent\textbf{Choice of local cross sections.}
By Proposition~\ref{pro:comcrosec} there is a finite family $\eS=\{S_k\}_{k=1}^L$ of pairwise disjoint sets $S_k$, each one being a topological two-manifold foliated by unstable leaves, being a local cross section of  time $\zeta$ for some $\zeta>0$ and of diameter at most $\alpha:=\delta$. Moreover, $\eS$ can be chosen such that $v\in{S_1}^\ast\subset S_1\subset D(v)$, that is, $S_1$ contains a neighborhood $U(v)$ of $v$ (relative to the induced topology on $S_1$).
Note that, by construction of $\eS$, we have $\sW^\u_{\delta/2}(v)\subset D(v)$.
\smallskip

\noindent\textbf{Intersection of $\Lambda_1$ with $\sW^\u_v$.}
Let $\varepsilon_1>0$ such that $B(v,\varepsilon_1)\cap S_1\subset U(v)$. 
By the well-known property of local product structure for any basic set, there exists $\delta_1<\delta$ such that for any $w$ in the basic set $\Lambda_1$ satisfying $\rho(w,v)\le \delta_1$ we have that the point $[w,v]$ is again in $\Lambda_1$, where $[\cdot,\cdot]$ is as defined in Lemma \ref{lem:locpro}. 
Since $\Lambda_1$ is a (nontrivial) basic set, none of its points is isolated and hence there exists $w\in\Lambda_1$ satisfying $\rho(w,v)\le\delta_1$. Hence, with these choices we can conclude that $[w,v]\in\sW^\u_{\varepsilon_1}(v)$. 

Note that, by analogous arguments, one can conclude that the segment between the points $v$ and $[w,v]$ which is contained in the local unstable manifold $\sW^\u_{\varepsilon_1}(v)$ contains further points of $\Lambda_1$ and hence of all further basic sets $\Lambda_\ell$, $\ell\ge2$.
\smallskip

\noindent\textbf{Choice of the SFT's and the one-sided cylinder $[a^{m_1}]\subset\Sigma_L$.}
Given the finite family of local cross sections $\eS$, consider the closed invariant set $\Sigma_Q\subset\Sigma_L$, and the hence associated increasing family of SFT's 
\[
	\Sigma^1\subset\Sigma^2\subset\ldots\subset\Sigma_L,
\]
where each $\Sigma^\ell$ is defined as in Sections~\ref{sec:loccro}--\ref{sec:symcodflobas} for $X=\Sigma^{\Lambda_\ell}$.  Let $\underline a,\underline b\in\Sigma^1$ such that $\Pi(\underline a)=v$ and $\Pi(\underline b)=w$. Since $w\in\sW^\u_\varepsilon(v)$ and since $\Pi|_{\Sigma^1}$ is a bijection, we have  $b_{-k}=a_{-k}$ for every $k\ge1$ and there exists a positive integer $m_1$ such that $b_k=a_k$ for every $k=0,\ldots,m_1-1$, that is, we have $\underline b\in[a^{m_1}]\eqdef [a_0\ldots a_{m_1-1}]$.
\smallskip

\noindent\textbf{Consideration of one-sided subshifts and choice of the equilibrium measures.}
Contrary to the simplifying notation in Section~\ref{sec:bridge}, we will now keep track of the one-sided aspect and use the notation ${}^+$.
Consider the natural projection $\pi^+\colon\Sigma_L\to\Sigma_L^+=\{0,\ldots,L-1\}^{\bN_0}$ given by $\pi^+(\ldots i_{-1}.i_0\ldots)\eqdef(i_0\ldots)$. Denote by 
\[
	\Sigma^{+,\ell}
	\eqdef \pi^+(\Sigma^\ell)
\]
the corresponding one-sided SFT's (which code the dynamics on the local unstable manifolds of the basic sets $\Lambda_\ell$).

For each $\ell\ge1$ let $\lambda_\ell\in\cM(G|_{\Lambda_\ell})$ be an equilibrium measure for some H\"older continuous potential as in the hypothesis of the proposition. Recalling again (see Section~\ref{sec:symcodflobas}) that $\Pi|_{\Sigma^\ell}\colon\Sigma^\ell\to\Lambda_\ell$ is a bijection, let $\nu_\ell\eqdef({\Pi|_{\Sigma^\ell}}^{-1})_\ast\lambda_\ell$.
Consider the Borel probability measure
\[
	\nu_\ell^+
	\eqdef(\pi^+)_\ast\nu_\ell.
\]

\noindent\textbf{Choice of the potential(s).}
Consider the continuous potential $\varphi^{(\u)}\colon T^1M\to\bR$ defined in~\eqref{phidef}.
Let 
\[
	\Delta=\Delta_{-\varphi^{(\u)}}\colon\Sigma_Q\to\bR
\]	 
denote the associated potential defined as in~\eqref{laut}. 
Consider the potential
\[
	\Delta^+\colon\Sigma_L^+\to\bR,
	\quad
	\Delta^+(i_0i_1\ldots)
	\eqdef \Delta(\ldots a_{-2}a_{-1}.i_0i_1\ldots).
\]
Recall that $\Delta$ and hence $\Delta^+$ are continuous.
\smallskip

This concludes the verification of the Hypothesis 1 and we let $\nu$ be the bridging measure with respect to $((\Sigma^{+,\ell})_\ell,(\nu_\ell^+)_\ell,\Delta^+,[a^{m_1}]^+)$ for some sufficiently rapidly growing sequence $(m_\ell)_\ell$ according to Proposition~\ref{p.pgfaweiss}. Recall that it is a Borel probability measure supported on $[a^{m_1}]^+$.
\smallskip

\noindent\textbf{Modeled counterpart.}
In the following, we will consider as modeled counterpart  $M$ the connected segment bounded by the points $[w,v]$ and $v$ of the unstable manifold $\sW^\u_v$. By Corollary~\ref{cor:unstmani}, for any  $\underline i\in[\ldots a_{-2}a_{-1}.a^{m_1}]$ we have $\Pi(\underline i)\in \sW^\u_v$. For each $\ell\ge1$ the coding map $\Pi$ restricted to $[a^{m_1}]$ is a continuous map. We define 
\[
	\Pi^+\colon [a^{m_1}]^+\to M,\quad
	\Pi^+(\underline i^+)
	\eqdef\Pi(\ldots a_{-2}a_{-1}.\underline i^+)
\]
which hence defines a continuous map.

By the above choice of $\Delta^+$, for any $\underline i\eqdef (\ldots a_{-2}a_{-1}.\underline i^+)$ we have
\begin{equation}\label{eq:before}\begin{split}
	L_m^+(\underline i^+)
	&= \frac1m\sum_{k=0}^{m-1}\Delta^+(\sigma^k(\underline i^+))
	= \frac1m\sum_{k=0}^{m-1}\Delta(\sigma^k(\underline i))	\\
	&=- \frac1m\int_0^{T_m(\underline i)}
		\varphi^{(\u)}(g^t(\Pi(\underline i))\,dt\\
	&=\frac1m\int_0^{T_m(\underline i)}
		\frac{d}{dt}\log\,\lVert dg^t|_{F^\u_{g^t(\Pi(\underline i))}}
			\rVert|_{t=0}\,dt\\	
	&= \frac1m\log \,\lVert dg^{T_m(\underline i)}
			|_{F^\u_{g^t(\Pi(\underline i))}}
			\rVert.
\end{split}\end{equation}
Hence, the limit inferior (superior) of $L_m^+$ coincides with the function $\underline \chi^+\circ\Pi$ (with $\overline\chi^+\circ\Pi$), where $\underline\chi^+$ ($\underline i^+$) is defined as in~\eqref{eq:defexponents}. 

Recall that each basic set $\Lambda_\ell$ has topological dimension one and that $\Lambda_\ell\cap S_1$ has topological dimension zero.
Let $N_\ell$ be an admissible length of the SFT $\Sigma^\ell$ and note that the family of sets $\{\Pi([i_{-kN_\ell}\ldots i_{kN_\ell}])\colon \underline i\in\Sigma^\ell\}$ provides a cover of $\Lambda_\ell\cap\cup\eS$ whose  diameter decreases when $k$ increases (see Lemma~\ref{lem:expansive}). Hence, without loss of generality, we can assume that  the family $\{\Pi^+([i_0\ldots i_{kN_\ell-1}]^+) \colon\underline i^+\in\Sigma^{+,\ell}\}$ is a family of pairwise disjoint sets. 

By hyperbolicity of $\Lambda_\ell$, local unstable manifolds have a uniform minimal size at any point in $\Lambda_\ell$. Hence, we have $\diam(\Pi([\underline i^-.]))\ge C_ \ell$ for any $\underline i\in\Sigma^\ell$ for some positive number $C_\ell$.

Note that any set $I\eqdef\Pi([\underline a^-.i_0\ldots i_{N-1}])=\Pi^+([i_0\ldots i_{N-1}]^+)$ is a connected segment of the local unstable  manifold $\sW^\u_v$ which contains $v=\Pi(\underline a)$. Recall again that the transition times between local cross sections are locally constant (Remark~\ref{rem:invtratime}) and hence $T_m(\underline a^-.\underline i^+)=T_m(\underline a)$ for any $\underline i^+\in [a^{m_1}]^+$. This, together with the fact that the geodesic flow is a factor of a suspension flow implies  (we denote $T=T_{kN_\ell}(\underline a)$ for short)
\[\begin{split}
	g^{T}\circ\Pi^+([ i_0\ldots i_{kN_\ell-1}]^+)
	&= g^{T}\circ p([\underline a^-.  i_0\ldots i_{kN_\ell-1}],0)\\
	\text{(property of the factor) }
	&= p\circ f^{T}([\underline a^-.  i_0\ldots i_{kN_\ell-1}],0)\\
	\text{(property of the suspension flow) }
	&= p(\sigma^{kT}([\underline a^-.  i_0\ldots i_{kN_\ell-1}]),0)\\
	\text{(definition of $p$) }
	&= \Pi(\sigma^{kT}([\underline a^-.  i_0\ldots i_{kN_\ell-1}]))\\
	&=\Pi([\underline a^-i_0\ldots i_{kN_\ell-1}.])	\\
\end{split}\]
By the above, the latter contains a segment $I$ of diameter at least $C_\ell>0$. Since the subbundle $F^\u$ is everywhere tangent to $\sW^\u$ and recalling the definition of $\varphi^{(\u)}$ in~\eqref{phidef}, for the diameter of $I$ and its image we have
\[\begin{split}
	&\diam_M(I) \\
	&	\phantom{w}
	\le  \diam_M(\Pi^+([i_0\ldots i_{kN_\ell-1}]^+)) 
		\Big(\max_{w\in \Pi^+([i_0\ldots i_{kN_\ell-1}]^+)}\,
			\lVert dg^{T_{kN_\ell}(\underline a)}|_{F^\u_w}\rVert\Big)\\
	&	\phantom{w}
	\le  \diam_M(\Pi^+([i_0\ldots i_{kN_\ell-1}]^+)) 
		 e^{\max_{w\in  \Pi^+([i_0\ldots i_{kN_\ell-1}]^+)} 
		 	-\int_0^{T_{kN_\ell}(\underline a)}\varphi^{(\u)}(g^t(w))\,dt}\\
	&\text{by~\eqref{eq:before} }
	\le  \diam_M(\Pi^+([i_0\ldots i_{kN_\ell-1}]^+)) 
		 e^{\max_{\underline i^+\in [i_0\ldots i_{kN_\ell-1}]^+}
		 	{kN_\ell L^+_{kN_\ell}(\underline i^+)}} \\
	&\text{by~\eqref{distorrrrtion} }
	\le  \diam_M(\Pi^+([i_0\ldots i_{kN_\ell-1}]^+)) 
		 e^{kN_\ell\rho^\ell_{kN_\ell}}
		 e^{{kN_\ell L^+_{kN_\ell}(\underline a^+)}}		 .
\end{split}\] 
This implies
\[
	\diam_M(\Pi^+([i_0\ldots i_{kN_\ell-1}]^+)) 
	\ge C_\ell e^{-kN_\ell\rho^\ell_{kN_\ell}}
		 e^{-{kN_\ell L^+_{kN_\ell}(\underline a^+)}}
\]
and hence we verified property~\eqref{H2:diameter} of Hypothesis 2.

This finishes the verification of Hypothesis 2.

Let $\lambda\eqdef\Pi_\ast\nu$ be the bridging measure for $((\Sigma^{+,\ell})_\ell,(\nu^+_\ell)_\ell,\Delta^+,[a^{m_1}]^+)$, the choice of $(m_\ell)_\ell$ and the modeled counterpart $\Pi^+\colon[a^{m_1}]^+\to M$.
By Proposition \ref{p.pgfaweiss-b}, there exists a $\lambda$-full measure set $L^+\subset M\subset \Pi^+([a^{m_1}]^+)\subset\sW^\u_v$ so that for every $x\in L^+$ we have $\underline d_\lambda(x)\ge\liminf_{\ell\to\infty}h_\ell/\Delta_\ell$ and and
\[
	\underline\chi(x)
	=\liminf_{\ell\to\infty}\Delta_\ell
	\quad\text{ and }\quad
	\overline\chi(x)
	=\limsup_{\ell\to\infty}\Delta_\ell.
\]

To estimate the Hausdorff dimension, we will now use the following result.

\begin{lemma}[{Non-uniform mass distribution principle~\cite[Proposition 2.1]{You:82}}]
	Let $Z\subset\bR^N$ be measurable and let $\mu$ be a finite non-atomic Borel measure on $\bR^n$ with $\mu(Z)>0$. Suppose that for every $x\in Z$ we have 
	\[
		D\le \underline d_\mu(x).
	\]
	Then $\dim_{\rm H}(Z)\ge D$.
\end{lemma}

This implies $\dim_{\rm H}(L^+)\ge \liminf_{\ell\to\infty}h_\ell/\Delta_\ell$, proving the proposition.
\end{proof}

We finally give a description of some (large) subset of the level set $\cL(0)=\cL^-(0)\cap\cL^+(0)$, which is an immediate consequence of the above proof.

\begin{corollary}\label{cor:product}
	For $v,L^-,L^+$ as in the assertion of Theorem \ref{jen}, we have
\[
	\{[w_1,w_2]\colon w_1\in L^-,w_2\in L^+\}
	\subset\cL(0)
\]	
\end{corollary}

\begin{proof}
	It suffices to observe that $L^+=\Pi^+(\widehat K^+)$, where $\widehat K^+\subset \Sigma_Q^+$ is some set of forward one-sided sequences which has full measure with respect to the bridging measure $\nu^+$ with respect to $((\Sigma^{+,\ell})_\ell,(\nu^+_\ell)_\ell,\Delta^+,[a^{m_1}]^+)$. The analogous construction can be done to obtain some set $L^-=\Pi^-(\widehat K^-)$ of backward one-sided sequences which has full measure with respect to the bridging measure $\nu^-$ with respect to $((\Sigma^{-,\ell})_\ell,(\nu^-_\ell)_\ell,\Delta^-,[a^{m_1}]^-)$. By the local product structure, we can well-define the set
	\[
		L
		\eqdef \{[\Pi^-(\underline i^-),\Pi^+(\underline i^+)]\colon\underline i^\pm\in\widehat K^\pm\}
	\]
	which has the claimed properties.
\end{proof}

\end{document}